\documentclass[11pt]{amsart}
\usepackage{eurosym}
\usepackage{amssymb}
\usepackage{amsfonts}
\usepackage{amsmath}
\usepackage{xcolor}

\newtheorem{theorem}{Theorem}[section]

\newtheorem{lemma}[theorem]{Lemma}
\theoremstyle{definition}

\theoremstyle{remark}
\newtheorem{remark}[theorem]{Remark}
\numberwithin{equation}{section}


\begin{document}
\title[singular Lane-Emden Neumann systems]{Existence and uniqueness results to a quasilinear singular Lane-Emden Neumann system}

\begin{abstract}
We establish the existence and uniqueness of solutions for quasilinear singular Lane-Emden type systems subjected to Neumann boundary conditions. The approach is chiefly based on sub-supersolutions method.
\end{abstract}

\author{Nouredine Medjoudj}
\address{Nouredine Medjoudj\\
	Applied Mathematics Laboratory, Faculty of Exact Sciences,\\
	A. Mira Bejaia University, Targa Ouzemour, 06000 Bejaia, Algeria}
\email{nouredine.medjoudj@univ-bejaia.dz}

\author{Abdelkrim Moussaoui}
\address{Abdelkrim Moussaoui\\
Applied Mathematics Laboratory (LMA), Faculty of Exact Sciences\\
and Department of Physico-Chemical Biology (BPC), Faculty of Natural and Life Sciences,\\
A. Mira Bejaia University, Algeria}
\email{abdelkrim.moussaoui@univ-bejaia.dz}

\maketitle

\section{Introduction}

\label{S1}

Let $\Omega $ be a bounded domain in $\mathbb{R}^{N}$ ($N\geq 2)$ having a
smooth boundary $\partial \Omega .$ Given $1<p_{i}<N$ for $i=1,2,$ we
consider the following quasilinear Lane-Emden type system 
\begin{equation*}
	\left( \mathrm{P}\right) \qquad \left\{ 
	\begin{array}{ll}
		-\Delta _{p_{1}}u+|u|^{p_{1}-2}u=u^{\alpha _{1}}+v^{\beta _{1}} & \text{in}%
		\;\Omega , \\ 
		-\Delta _{p_{2}}v+|v|^{p_{2}-2}v=u^{\alpha _{2}}+v^{\beta _{2}} & \text{in}%
		\;\Omega , \\ 
		u,v>0 & \text{in }%
		\;\Omega ,\\
		\frac{\partial u}{\partial \eta }=\frac{\partial v}{\partial \eta }=0 & 
		\text{on}\;\partial \Omega ,%
	\end{array}%
	\right.
\end{equation*}%
where $\eta $ is the unit outer normal to $\partial \Omega $, while $\Delta
_{p_{i}}$ denotes the $p_{i}$-Laplace operator, namely $\Delta _{p_{i}}:=%
\mathrm{div}(|\nabla w|^{p_{i}-2}\nabla w)$, $\forall \,w\in
W^{1,p_{i}}(\Omega ).$ We consider system $\left( \mathrm{P}\right) $ in
a singular case assuming that the exponents verify the condition%
\begin{equation}
	-1<\alpha _{1},\beta _{2}<0,\text{ \ }-1<\beta _{1}<p_{1}-1\text{ \ and \ }%
	-1<\alpha _{2}<p_{2}-1.  \label{c1}
\end{equation}

A solution $(u,v)\in W^{1,p_{1}}(\Omega )\times W^{1,p_{2}}(\Omega )$ of $%
\left( \mathrm{P}\right) $ is understood in the weak sense, that is,%
\begin{equation}
	\left\{ 
	\begin{array}{l}
		\int_{\Omega }|\nabla u|^{p_{1}-2}\nabla u\nabla \varphi \,\mathrm{d}%
		x+\int_{\Omega }|u|^{p_{1}-2}u\varphi \,\mathrm{d}x=\int_{\Omega }(u^{\alpha
			_{1}}+v^{\beta _{1}})\varphi \,\mathrm{d}x \\ 
		\int_{\Omega }|\nabla v|^{p_{2}-2}\nabla v\nabla \psi \,\mathrm{d}%
		x+\int_{\Omega }|v|^{p_{2}-2}v\,\psi \,\mathrm{d}x=\int_{\Omega }(u^{\alpha
			_{2}}+v^{\beta _{2}})\,\psi \,\mathrm{d}x%
	\end{array}%
	\right.  \label{DF01}
\end{equation}%
for all $(\varphi ,\psi )\in W^{1,p_{1}}(\Omega )\times W^{1,p_{2}}(\Omega )$%
.

System $\left( \mathrm{P}\right) $ is a natural extension and generalization
of the celebrated Lane-Emden equation%
\begin{equation}
	\Delta w+w^{\gamma }=0\text{ in }\Omega ,  \label{LE}
\end{equation}%
subjected to Neumann boundary conditions. It is introduced by Homer Lane \cite{Lane}, who was interested in computing both the temperature and the density of mass on the surface of the sun. The Lane-Emden equation (\ref{LE}) has been the focus of a huge number of works which we shall not discuss merely
mentioning some of them \cite{ALVR, Ben, F}. It is involved in wide range of
phenomena in mathematical physics and chemistry, specifically in the areas
of conformal geometry, thermal explosion, isothermal gas spheres and
thermionic currents \cite{S}. In astrophysics, it describes the behavior of
the density of a gas sphere in hydrostatic equilibrium. The index $\gamma ,$
called the polytropic index, is related to the ratio $r$ of the specific
heats of the gas throught $\gamma =\frac{1}{r-1}.$ It is used to determine
the structure of the interior of polytropic stars, which are subject to the
influence of their own gravitational field \cite{C}. The case $\gamma <0$ in
(\ref{LE}) is highly challenging as it brings out a singularity at the
origin. It has attracted considerable interest in recent years. We mention \cite{CRT} where it is shown that problem (\ref{LE}) subject to Dirichlet boundary conditions admits a unique positive solution $u$ in $\mathcal{C}^{2}(\Omega )\cap \mathcal{C}(\overline{\Omega })$. In addition, this solution $u$ belongs to $\mathcal{C}^{2}(\Omega )\cap \mathcal{C}^{1}(\overline{\Omega })$ once $-1<\gamma <0.$ When the quasilinear $p$-Laplacian operator is involved in singular problem (\ref{LE}), \cite{GST} provides a unique solution in  $\mathcal{C}^{1,\tau }(\overline{\Omega })$, $\tau \in (0,1),$ for $-1<\gamma <0$.

Singularities are an important feature in the study of the problem $\left( 
\mathrm{P}\right) $. They occur near the origin under assumption (\ref{c1}) on nonlinearities. This fact represents a serious difficulty
to overcome especially since a very marked singularity character is
considered for $\left( \mathrm{P}\right) $, resulting from (\ref{c1}) when
all exponents are negative. Actually, singularities are
involved in a wide range of important elliptic problems which have been
studied extensively in recent years, see for instance \cite{AC, CRT, DM1,
	DM2, G, GHM, GST, H, KM, MM3, MM2, MM1, MKT, M1} and there references. The
semilinear case, that is when $p_{1}=p_{2}=2$, has been widely investigated, especially in the context of the Gierer Meinhardt system (see, e.g., \cite{M1} and the references therein). However, as far as we know, \cite{GM} is the only paper where Neumann boundary conditions is considered for singular quasilinear systems. We emphasize that system $\left( \mathrm{P}\right) $ cannot be incorporated neither in Gierer-Meinhardt system addressed in \cite{M1} nor in the convective system studied in \cite{GM} even if gradient terms involved are canceled. 

In the aforementioned papers, two complementary structures are separately discussed. In \cite{AC, DM1, DM2, GHM, KM,
	MM1}, the systems examined are cooperative while in \cite{G, MKT,  MM3, MM2}, they are competitive. For system $\left( \mathrm{P}\right) $, these both structures are closely related to the sign of exponents $\alpha _{2}$ and $\beta _{1}$. Namely, $\left( \mathrm{P}\right) $ is cooperative if $\min\{\alpha _{2},\beta _{1}\}>0$ whereas for opposite strict inequality, $\left( \mathrm{P}\right) $ is competitive. It should be noted that, in the present paper, these both complementary structures for the system $\left(\mathrm{P}\right) $ are handled simultaneously without referring to them despite their important structural disparity that makes the right hand side in $\left( \mathrm{P}\right) $ behaving in a drastically different way.

Motivated by the aforementioned facts, our main concern is the question of
existence of solutions for system $\left( \mathrm{P}\right) $. By
adequate truncation and owing to Schauder's fixed point theorem, we first
develop two new sub-supersolutions Theorems for a general singular Neumann
Systems involving $p$-Laplacian operator (cf. Theorems \ref{T2} and \ref{T4}%
, section \ref{S3}). These results address respectively the case of bounded
and unbounded nonlinearities, restricted to the rectangle formed by the
sub-supersolution pair. Thence, the nature and properties of the
sub-supersolutions are meaningfully impacted which, henceforth, can be
constructed on a wider choice of functions than specified in the literature.
Namely, this enable to consider subsolutions subject to homogeneous Dirichlet boundary conditions, what was not possible in \cite{GM} under assumption $($\textrm{H}$)$ stated therein. It is worth noting that, unlike to what has been stated in \cite{M1}, sub-supersolutions cannot have both zero traces on the boundary $\partial \Omega $. This would lead to a solution for problem $%
\left( \mathrm{P}\right) $ with zero trace condition on $\partial \Omega $
and therefore, according to \cite[Lemma 3.1]{H}, its normal derivative would
be nontrivial, which is absurd. Another important issue being addressed by
Theorems \ref{T2} and \ref{T4} concern the property sign
of normal derivative on $\partial \Omega $ of sub-supersolutions. It is
established that the condition (\ref{c1*}) is crucial to handle quasilinear
Neumann problems via sub-supersolutions method and therefore, in no case
could be ignored. Furthermore, in the aforementioned Theorems, we mention
that no sign condition is required on the right-hand side nonlinearities.
Hence, they can be used for large classes of quasilinear singular problems,
including those discussed in \cite{GM, M1}.

In the prospect of applying Theorems \ref{T2} and \ref{T4}, pairs of
sub-supersolutions for system $\left( \mathrm{P}\right) $ are constructed by
exploiting spectral properties of the $p$-Laplacian operator as well as
properties of torsion problems, both subject to Dirichlet or Neumann
boundary conditions (cf. section \ref{S2}). A suitable adjustment of constants is also required. Moreover, by making some specific and necessary adjustments, it is quite possible to construct a pair of sub-supersolution
in the spirit of \cite{GM} which verifies assumption (\ref{c1*}).
Nevertheless, although their new form is appropriate, it no longer leads to
an infinite sequence of solutions, as it is stated in \cite{GM}. This seems
to be unfeasible under assumption (\ref{c1*}).

The sub-supersolutions pairs establish a location of solutions of $\left( 
\mathrm{P}\right) $ provided by Theorems \ref{T7} and \ref{T8} (cf. section \ref{S4}). We show that their regularity property depend on the behavior of the subsolution near the boundary $\partial \Omega $. Precisely, when the subsolution has a zero trace condition on $\partial \Omega ,$ solutions $(u,v)$ of $\left( \mathrm{P}\right) $ are bounded in $L^{\infty }(\Omega
)\times L^{\infty }(\Omega )$, whereas in the complementary case, $(u,v)$ are
bounded in $\mathcal{C}^{1}(\overline{\Omega })\times \mathcal{C}^{1}(%
\overline{\Omega })$. Note that the absence of $\mathcal{C}^{1}$-bound for
solutions in the previous case is rather due to the fact that we have not been able to find in the literature an equivalent of the regularity theorem for singular Dirichlet problems \cite[Lemma 3.1]{H} to the case of Neumann boundary conditions. Regularity results for singular Neumann problems remains an open question.

A second main objective of our work is to provide a criterion ensuring the
uniqueness of a solution to problem $\left( \mathrm{P}\right) $. In this
direction, the uniqueness result is established for $\mathcal{C}^{1}$-bound solutions. The proof is based on the adaption of argument by Krasnoselskii \cite{K} where properties of sub-supersolutions are crucial. In this part of our work, we assume all exponents in (\ref{c1}) negative, subject to some additional restrictions as described in our result (cf. section \ref{S4}). It is worth noting that the previous argument does not apply to $L^{\infty}$- bound solutions simply because the obtained supersolutions are not comparable to the distance function $d(x)$ in $\Omega$. 

The rest of this article is organized as follows. Section \ref{S3} contains
the existence theorems involving sub-supersolutions. Section \ref{S4}
presents abstract existence and uniqueness results. Section \ref{S2} deals
with the existence and uniqueness of solutions for problem $\left( \mathrm{P}%
\right) $.

\section{Sub-super-solution theorems}

\label{S3}

The Sobolev spaces $W^{1,p_{i}}(\Omega )$ and $W_{0}^{1,p_{i}}(\Omega )$
will be equipped with the norm 
\begin{equation*}
	\Vert w\Vert _{1,p_{i}}:=\left( \Vert w\Vert _{p_{i}}^{p_{i}}+\Vert \nabla
	w\Vert _{p_{i}}^{p_{i}}\right) ^{\frac{1}{p_{i}}},\quad w\in
	W^{1,p_{i}}(\Omega ),
\end{equation*}%
\begin{equation*}
	\Vert w\Vert _{1,p_{i}}:=\Vert \nabla w\Vert _{p_{i}},\quad w\in
	W_{0}^{1,p_{i}}(\Omega ),
\end{equation*}%
where, as usual, 
\begin{equation*}
	\Vert w\Vert _{p_{i}}:=\left\{ 
	\begin{array}{ll}
		\left( \int_{\Omega }|w(x)|^{p_{i}}\mathrm{d}x\right) ^{\frac{1}{p_{i}}} & 
		\text{ if }p_{i}<+\infty , \\ 
		&  \\ 
		\underset{x\in \Omega }{ess\sup }\,|w(x)| & \text{ otherwise.}%
	\end{array}%
	\right. 
\end{equation*}%
We denote by $(W^{1,p_{i}}(\Omega ))^{\ast }$ the topological dual space of $%
W^{1,p_{i}}(\Omega )$. Moreover, 
\begin{equation*}
	W_{+}^{1,p_{i}}(\Omega ):=\{w\in W^{1,p_{i}}(\Omega ):w\geq 0\},\quad
	W_{b}^{1,p_{i}}(\Omega ):=W^{1,p_{i}}(\Omega )\cap L^{\infty }(\Omega ),
\end{equation*}%
and%
\begin{equation*}
	W_{0,b}^{1,p_{i}}(\Omega ):=W_{0}^{1,p_{i}}(\Omega )\cap L^{\infty }(\Omega )
\end{equation*}%
We also utilize%
\begin{eqnarray*}
	\mathcal{C}_{+}^{1,\tau}(\overline{\Omega }) &:&=\{w\in \mathcal{C}^{1,\tau}(\overline{\Omega }):w\geq 0\text{ for all }x\in \overline{\Omega }\}, \\
	int\mathcal{C}_{+}^{1,\tau}(\overline{\Omega }) &:&=\{w\in \mathcal{C}_{+}^{1,\tau}(%
	\overline{\Omega }):w>0\text{ for all }x\in \overline{\Omega }\}, \text{ for } \tau\in (0,1).
\end{eqnarray*}%
In what follows, we set $r^{\pm }:=\max \{\pm r,0\}$ and we denote by $%
\gamma _{0}$ the unique continuous linear map $\gamma
_{0}:W^{1,p_{i}}(\Omega )\rightarrow L^{p_{i}}(\partial \Omega )$ known as
the trace map such that $\gamma _{0}(u)=u_{/\partial \Omega },$ for all $%
u\in W^{1,p}(\Omega )$ and verifies the property 
\begin{equation}
	\gamma _{0}(u^{+})=\gamma _{0}(u)^{+}\text{ \ for all }u\in
	W^{1,p_{i}}(\Omega ),  \label{1}
\end{equation}%
(see, e.g., \cite{MMP}).

Hereafter, we denote by $d(x)$ the distance from a point $x\in \overline{%
	\Omega }$ to the boundary $\partial \Omega $, where $\overline{\Omega }%
=\Omega \cup \partial \Omega $ is the closure of $\Omega \subset 
\mathbb{R}
^{N}$. For $1<r<N$ and $-r<s\leq 0,$ it is known that 
\begin{equation*}
	\left( \int_{\Omega }d(x)^{s}|u(x)|^{r}\mathrm{d}x\right) ^{\frac{1}{r}}\leq
	C\Vert u\Vert _{1,r}\quad \forall \,u\in W^{1,r}(\Omega ),
\end{equation*}%
with suitable $C>0$; see \cite[Theorem 19.9, case (19.29)]{OK}. Accordingly,
by H\"{o}lder's inequality, if $-1<\beta \leq 0$ then 
\begin{equation}
	\int_{\Omega }|d(x)^{\beta } u(x)|\,\mathrm{d}x\leq |\Omega |^{\frac{1}{%
			r^{\prime }}}\left( \int_{\Omega }d(x)^{\beta r}|u(x)|^{r}\mathrm{d}x\right) ^{%
		\frac{1}{r}}\leq C|\Omega |^{\frac{1}{r^{\prime }}}\Vert u\Vert
	_{1,r},\;\;u\in W^{1,r}(\Omega ).  \label{54}
\end{equation}

Finally, we say that $j:\Omega \times \mathbb{R}^{2}\rightarrow \mathbb{R}$
is a Carath\'{e}odory function provided

\begin{itemize}
	\item[$\bullet $] $x\mapsto j(x,s,t)$ is measurable for every $(s,t)\in 
	\mathbb{R}^{2}$, and
	
	\item[$\bullet $] $(s,t)\mapsto j(x,s,t)$ is continuous for almost all $x\in
	\Omega $.
\end{itemize}

\mathstrut

This section investigates the existence of solutions to system%
\begin{equation*}
	\left( \mathrm{P}_{f_{1},f_{2}}\right) \qquad \left\{ 
	\begin{array}{ll}
		-\Delta _{p_{1}}u+|u|^{p_{1}-2}u=f_{1}(x,u,v) & \text{in}\;\Omega , \\ 
		-\Delta _{p_{2}}v+|v|^{p_{2}-2}v=f_{2}(x,u,v) & \text{in}\;\Omega , \\ 
		\frac{\partial u}{\partial \eta }=\frac{\partial v}{\partial \eta }=0 & 
		\text{on}\;\partial \Omega ,%
	\end{array}%
	\right.
\end{equation*}%
where $f_{i}:\Omega \times \mathbb{R}^{2}\rightarrow \mathbb{R}$ satisfy
Carath\'{e}odory's conditions. The following assumptions will be posited.

\begin{itemize}
	\item[$(\mathrm{H}_{1})$] With appropriate $(\underline{u},\underline{v}),(%
	\overline{u},\overline{v})\in W_{b}^{1,p_{1}}(\Omega )\times
	W_{b}^{1,p_{2}}(\Omega )$ such that%
	\begin{equation}
		\max \{\frac{\partial \underline{u}}{\partial \eta },\frac{\partial 
			\underline{v}}{\partial \eta }\}\leq 0\leq \min \{\frac{\partial \overline{u}%
		}{\partial \eta },\frac{\partial \overline{v}}{\partial \eta }\}\text{ on }%
		\partial \Omega ,  \label{c1*}
	\end{equation}%
	one has $\underline{u}\leq \overline{u}$, $\underline{v}\leq \overline{v}$,
	as well as 
	\begin{equation}
		\left\{ 
		\begin{array}{l}
			\int_{\Omega }|\nabla \underline{u}|^{p_{1}-2}\nabla \underline{u}\nabla
			\varphi \,\mathrm{d}x+\int_{\Omega }|\underline{u}|^{p_{1}-2}\underline{u}%
			\varphi \,\mathrm{d}x \\ 
			-\int_{\partial \Omega }|\nabla \underline{u}|^{p_{1}-2}\frac{\partial 
				\underline{u}}{\partial \eta }\gamma _{0}(\varphi )\text{ }\mathrm{d}s\leq
			\int_{\Omega }f_{1}(\cdot ,\underline{u},v)\varphi \,\mathrm{d}x, \\ 
			\\ 
			\int_{\Omega }|\nabla \underline{v}|^{p_{2}-2}\nabla \underline{v}\,\nabla
			\psi \,\mathrm{d}x+\int_{\Omega }|\underline{v}|^{p_{2}-2}\underline{v}%
			\,\psi \,\mathrm{d}x \\ 
			-\int_{\partial \Omega }|\nabla \underline{v}|^{p_{2}-2}\frac{\partial 
				\underline{v}}{\partial \eta }\gamma _{0}(\psi )\text{ }\mathrm{d}s\leq
			\int_{\Omega }f_{2}(\cdot ,u,\underline{v})\psi \,\mathrm{d}x,%
		\end{array}%
		\right.  \label{c2}
	\end{equation}%
	\begin{equation}
		\left\{ 
		\begin{array}{l}
			\int_{\Omega }|\nabla \overline{u}|^{p_{1}-2}\nabla \overline{u}\,\nabla
			\varphi \,\mathrm{d}x+\int_{\Omega }|\overline{u}|^{p_{1}-2}\overline{u}%
			\,\varphi \,\mathrm{d}x \\ 
			-\int_{\partial \Omega }|\nabla \overline{u}|^{p_{1}-2}\frac{\partial 
				\overline{u}}{\partial \eta }\gamma _{0}(\varphi )\text{ }\mathrm{d}s\geq
			\int_{\Omega }f_{1}(\cdot ,\overline{u},v)\varphi \,\mathrm{d}x, \\ 
			\\ 
			\int_{\Omega }|\nabla \overline{v}|^{p_{2}-2}\nabla \overline{v}\,\nabla
			\psi \,\mathrm{d}x+\int_{\Omega }|\overline{v}|^{p_{2}-2}\overline{v}\,\psi
			\,\mathrm{d}x \\ 
			-\int_{\partial \Omega }|\nabla \overline{v}|^{p_{2}-2}\frac{\partial 
				\overline{v}}{\partial \eta }\gamma _{0}(\psi )\text{ }\mathrm{d}s\geq
			\int_{\Omega }f_{2}(\cdot ,u,\overline{v})\psi \,\mathrm{d}x%
		\end{array}%
		\right.  \label{c3}
	\end{equation}%
	for all $(\varphi ,\psi )\in W_{+}^{1,p_{1}}(\Omega )\times
	W_{+}^{1,p_{2}}(\Omega )$, $(u,v)\in W^{1,p_{1}}(\Omega )\times
	W^{1,p_{2}}(\Omega )$ such that $(u,v)\in \lbrack \underline{u},\overline{u}%
	]\times \lbrack \underline{v},\overline{v}]$.
	
	\item[$(\mathrm{H}_{2})$] For appropriate $C>0$ and $\gamma \in (-1,0)$ one
	has 
	\begin{equation*}
		\max |f_{i}(x,s,t)|\leq Cd(x)^{\gamma }\quad \text{in}\quad \Omega \times
		\lbrack \underline{u},\overline{u}]\times \lbrack \underline{v},\overline{v}%
		].
	\end{equation*}
\end{itemize}

Note that under $(\mathrm{H}_{2})$ and by virtue of the Hardy-Sobolev inequality type in (\ref{54}), the integrals involving $f_{1}$ and $f_{2}$ in (\ref{c2}) and (\ref{c3}) take sense .

\begin{theorem}
	\label{T2} Suppose $(\mathrm{H}_{1})$--$(\mathrm{H}_{2})$ hold true. Then,
	problem $(\mathrm{P}_{f_{1},f_{2}})$ possesses a solution $(u,v)\in
	W_{b}^{1,p_{1}}(\Omega )\times W_{b}^{1,p_{2}}(\Omega )$ such that 
	\begin{equation}
		\underline{u}\leq u\leq \overline{u}\quad \text{and}\quad \underline{v}\leq
		v\leq \overline{v}.  \label{19}
	\end{equation}%
	Moreover, $\frac{\partial u}{\partial \eta }=\frac{\partial v}{\partial \eta 
	}=0$ on $\partial \Omega $.
\end{theorem}

\begin{proof}
	Given $(z_{1},z_{2})\in L^{p_{1}}(\Omega )\times L^{p_{2}}(\Omega )$, we
	define 
	\begin{equation}
		\mathrm{T}_{1}(z_{1}):=\left\{ 
		\begin{array}{ll}
			\underline{u} & \text{when }z_{1}\leq \underline{u}, \\ 
			z_{1} & \text{if }\underline{u}\leq z_{1}\leq \overline{u}, \\ 
			\overline{u} & \text{otherwise,}%
		\end{array}%
		\right. \quad \mathrm{T}_{2}(z_{2}):=\left\{ 
		\begin{array}{ll}
			\underline{v} & \text{when }z_{2}\leq \underline{v}, \\ 
			z_{2} & \text{if }\underline{v}\leq z_{2}\leq \overline{v}, \\ 
			\overline{v} & \text{otherwise.}%
		\end{array}%
		\right.  \label{0}
	\end{equation}%
	Assumption $(\mathrm{H}_{2})$ together with Hardy-Sobolev type inequality (\ref{54}) we
	infer that 
	\[
	f_{i}(x,\mathrm{T}_{1}(z_{1}),\mathrm{T}_{2}(z_{2}))\in
	\left(W^{1,p_{i}}(\Omega ) \right)^{*}, \text{ for } i=1,2.
	\]%
	Then, from Minty-Browder Theorem (see, e.g., \cite[Theorem V.15]{B}), it
	follows that the auxiliary problem 
	\begin{equation}
		\left\{ 
		\begin{array}{ll}
			-\Delta _{p_{1}}u+|u|^{p_{1}-2}u=f_{1}(x,\mathrm{T}_{1}(z_{1}),\mathrm{T}%
			_{2}(z_{2})) & \text{in }\Omega , \\ 
			-\Delta _{p_{2}}v+|v|^{p_{2}-2}v=f_{2}(x,\mathrm{T}_{1}(z_{1}),\mathrm{T}%
			_{2}(z_{2})) & \text{in }\Omega , \\ 
			\frac{\partial u}{\partial \eta }=\frac{\partial v}{\partial \eta }=0 & 
			\text{on }\partial \Omega ,%
		\end{array}%
		\right.  \label{301}
	\end{equation}%
	admits a unique solution $(u,v)\in W^{1,p_{1}}(\Omega )\times
	W^{1,p_{2}}(\Omega ).$
	
	Let us introduce the operator 
	\[
	\begin{array}{cccc}
		\mathcal{T}: & L^{p_{1}}(\Omega )\times L^{p_{2}}(\Omega ) & \rightarrow & 
		W^{1,p_{1}}(\Omega )\times W^{1,p_{2}}(\Omega ) \\ 
		& (z_{1},z_{2}) & \mapsto & (u,v).%
	\end{array}%
	\]%
	We note from (\ref{301}) that any fixed point of $\mathcal{T}$ within $[%
	\underline{u},\overline{u}]\times \lbrack \underline{v},\overline{v}]$
	coincides with the weak solution of $(\mathrm{P}_{f_{1},f_{2}})$.
	
	Let us show that $\mathcal{T}$ is continuous. Let $(z_{1,n},z_{2,n})%
	\rightarrow (z_{1},z_{2})$ in $L^{p_{1}}(\Omega )\times L^{p_{2}}(\Omega )$
	for all $n$. Denote $\left( u_{n},v_{n}\right) =\mathcal{T}(z_{1,n},z_{2,n})$%
	, which reads as 
	\begin{equation}
		\begin{array}{c}
			\int_{\Omega }|\nabla u_{n}|^{p_{1}-2}\nabla u_{n}\nabla \varphi _{1}\,%
			\mathrm{d}x+\int_{\Omega }|u_{n}|^{p_{1}-2}u_{n}\varphi _{1}\,\mathrm{d}x \\ 
			=\int_{\Omega }f_{1}(x,\mathrm{T}_{1}(z_{1,n}),\mathrm{T}_{2}(z_{2,n}))%
			\varphi _{1}\,\mathrm{d}x%
		\end{array}
		\label{310}
	\end{equation}%
	and 
	\begin{equation}
		\begin{array}{c}
			\int_{\Omega }|\nabla v_{n}|^{p_{2}-2}\nabla v_{n}\nabla \varphi _{2}\,%
			\mathrm{d}x+\int_{\Omega }|v_{n}|^{p_{2}-2}v_{n}\varphi _{2}\,\mathrm{d}x \\ 
			=\int_{\Omega }f_{2}(x,\mathrm{T}_{1}(z_{1,n}),\mathrm{T}_{2}(z_{2,n}))%
			\varphi _{2}\,\mathrm{d}x%
		\end{array}
		\label{311}
	\end{equation}%
	for all $\varphi _{i}\in W^{1,p_{i}}(\Omega ),\ i=1,2$. Inserting $(\varphi
	_{1},\varphi _{2})=(u_{n},v_{n})$ in (\ref{310}) and (\ref{311}), using $(%
	\mathrm{H}_{2})$ and (\ref{54}), we get 
	\begin{equation}
		\left\Vert u_{n}\right\Vert _{1,p_{1}}^{p_{1}}=\int_{\Omega }f_{1}(x,\mathrm{%
			T}_{1}(z_{1,n}),\mathrm{T}_{2}(z_{2,n}))u_{n}\,\mathrm{d}x\leq C\int_{\Omega
		}d(x)^{\gamma }u_{n}\,\mathrm{d}x  \label{306}
	\end{equation}%
	and 
	\begin{equation}
		\left\Vert v_{n}\right\Vert_{1,p_{2}}^{p_{2}}=\int_{\Omega }f_{2}(x,\mathrm{T%
		}_{1}(z_{1,n}),\mathrm{T}_{2}(z_{2,n}))v_{n}\,\mathrm{d}x\leq C\int_{\Omega
		}d(x)^{\gamma }v_{n}\,\mathrm{d}x.  \label{307}
	\end{equation}%
	Since $-1<\gamma <0$, by virtue of the Hardy-Sobolev inequality type in (\ref{54}), the last integrals in (\ref{306}) and (\ref{307}) are finite and it holds
	\begin{equation}
		\left\Vert u_{n}\right\Vert _{1,p_{1}}^{p_{1}} \leq C_{1} \left\Vert
		u_{n}\right\Vert _{1,p_{1}}  \label{306-}
	\end{equation}%
	and 
	\begin{equation}
		\left\Vert v_{n}\right\Vert_{1,p_{2}}^{p_{2}} \leq C_{1} \left\Vert
		v_{n}\right\Vert_{1,p_{2}},  \label{307-}
	\end{equation}
	for a certain $C_1>0$ independent of $n$. Thus, $\{u_{n}\}$ and $\{v_{n}\}$ are bounded in $%
	W^{1,p_{1}}(\Omega )$ and $W^{1,p_{2}}(\Omega )$, respectively. So, passing
	to relabeled subsequences, we get 
	\begin{equation}
		\begin{array}{c}
			(u_{n},v_{n})\rightharpoonup \left( u,v\right) \text{ in }W^{1,p_{1}}(\Omega
			)\times W^{1,p_{2}}(\Omega ),%
		\end{array}
		\label{312}
	\end{equation}%
	for certain $\left( u,v\right) $ in $W^{1,p_{1}}(\Omega )\times
	W^{1,p_{2}}(\Omega )$. Setting $\varphi _{1}=u_{n}-u$ in (\ref{310}) we find
	that 
	\[
	\begin{array}{l}
		\int_{\Omega }|\nabla u_{n}|^{p_{1}-2}\nabla u_{n}\nabla (u_{n}-u)\,\mathrm{d%
		}x+\int_{\Omega }|u_{n}|^{p_{1}-2}u_{n}(u_{n}-u)\,\mathrm{d}x \\ 
		=\int_{\Omega }f_{1}(x,\mathrm{T}_{1}(z_{1,n}),\mathrm{T}%
		_{2}(z_{2,n}))(u_{n}-u)\ \mathrm{d}x%
	\end{array}%
	\]%
	Note that $(\mathrm{H}_{2})$ as well as (\ref{54}) ensure that 
	\begin{equation}
		f_{1}(x,\mathrm{T}_{1}(z_{1,n}),\mathrm{T}_{2}(z_{2,n}))(u_{n}-u)\in
		L^{1}(\Omega ).  \label{300}
	\end{equation}%
	Thus, Fatou's Lemma implies%
	\[
	\begin{array}{l}
		\underset{n\rightarrow \infty }{\lim }\sup \int_{\Omega }f_{1}(x,\mathrm{T}%
		_{1}(z_{1,n}),\mathrm{T}_{2}(z_{2,n}))(u_{n}-u)\ dx \\ 
		\leq \int_{\Omega }\underset{n\rightarrow \infty }{\lim }\sup \left( f_{1}(x,%
		\mathrm{T}_{1}(z_{1,n}),\mathrm{T}_{2}(z_{2,n}))(u_{n}-u)\right) \
		dx\rightarrow 0,%
	\end{array}%
	\]%
	showing that%
	\[
	\underset{n\rightarrow \infty }{\lim }\sup \left\langle -\Delta
	_{p_{1}}u_{n}+|u_{n}|^{p_{1}-2}u_{n},u_{n}-u\right\rangle \leq 0.
	\]%
	Likewise, we prove that%
	\[
	\underset{n\rightarrow \infty }{\lim }\sup \left\langle -\Delta
	_{p_{2}}v_{n}+|v_{n}|^{p_{2}-2}v_{n},v_{n}-v\right\rangle \leq 0.
	\]
	
	Then, the $S_{+}$-property of $-\Delta _{p_{i}}$ on $W^{1,p_{i}}(\Omega )$
	(see, e.g., \cite[Proposition 2.72]{MMP}) along with (\ref{312}) implies 
	\begin{equation}
		\begin{array}{c}
			(u_{n},v_{n})\rightarrow (u,v)\text{ in }W^{1,p_{1}}(\Omega )\times
			W^{1,p_{2}}(\Omega ).%
		\end{array}
		\label{9}
	\end{equation}%
	Through (\ref{310}), (\ref{311}) and the invariance of $L^{p_{1}}(\Omega
	)\times L^{p_{2}}(\Omega )$ by $\mathcal{T}$, we infer that $\left(
	u,v\right) =\mathcal{T}(z_{1},z_{2})$, as desired.
	
	Let us verify that $\mathcal{T}(L^{p_{1}}(\Omega )\times L^{p_{2}}(\Omega ))$
	is a relatively compact subset. If $(u_{n},v_{n}):=\mathcal{T}%
	(y_{1,n},y_{2,n}),$ $n\in 
	\mathbb{N}
	$, (\ref{310}) and (\ref{311}) can be written. Hence, the previous argument
	yields a pair $(u,v)\in W^{1,p_{1}}(\Omega )\times W^{1,p_{2}}(\Omega )$
	fulfilling (\ref{9}), possibly along a subsequence.
	
	We are thus in a position to apply Schauder's fixed point theorem to the map 
	$\mathcal{T}$, which establishes the existence of $(u,v)\in
	W^{1,p_{1}}(\Omega )\times W^{1,p_{2}}(\Omega )$ satisfying $(u,v)=\mathcal{T%
	}(u,v).$ Due to \cite[Theorem 3]{CF}, one has 
	\[
	\frac{\partial u}{\partial \eta }=\frac{\partial v}{\partial \eta }=0\;\;%
	\text{on }\partial \Omega .
	\]%
	Hence, $(u,v)$ is a solution of $(\mathrm{P}_{f_{1},f_{2}})$. Let us show
	that (\ref{19}) is fulfilled. Put $\zeta =(\underline{u}-u)^{+}$ and suppose 
	$\zeta \neq 0$. Then, from $(\mathrm{H}_{1})$, (\ref{0}) and (\ref{301}) ,
	we get 
	\[
	\begin{array}{l}
		\int_{\{u<\underline{u}\}}|\nabla u|^{p_{1}-2}\nabla u\nabla \zeta \ \mathrm{%
			d}x+\int_{\{u<\underline{u}\}}|u|^{p_{1}-2}u\zeta \ \mathrm{d}x \\ 
		=\int_{\Omega }|\nabla u|^{p_{1}-2}\nabla u\nabla \zeta \ \mathrm{d}%
		x+\int_{\Omega }|u|^{p_{1}-2}u\zeta \ \mathrm{d}x=\int_{\Omega }f_{1}(x,%
		\mathrm{T}_{1}(u),\mathrm{T}_{2}(v))\zeta \ \mathrm{d}x \\ 
		=\int_{\{u<\underline{u}\}}f_{1}(x,\mathrm{T}_{1}(u),\mathrm{T}_{2}(v))\zeta
		\ \mathrm{d}x=\int_{\{u<\underline{u}\}}f_{1}(x,\underline{u},\mathrm{T}%
		_{2}(v))\zeta \ \mathrm{d}x \\ 
		\geq \int_{\{u<\underline{u}\}}|\nabla \underline{u}|^{p_{1}-2}\nabla 
		\underline{u}\nabla \zeta \ \mathrm{d}x+\int_{\{u<\underline{u}\}}|%
		\underline{u}|^{p_{1}-2}\underline{u}\zeta \ \mathrm{d}x-\int_{\partial
			\Omega }|\nabla \underline{u}|^{p_{1}-2}\frac{\partial \underline{u}}{%
			\partial \eta }\gamma _{0}(\zeta )\text{ }\mathrm{d}s.%
	\end{array}%
	\]%
	which by $(\mathrm{H}_{1})$ and (\ref{1}), implies that 
	\[
	\begin{array}{c}
		\int_{\{u<\underline{u}\}}(|\nabla \underline{u}|^{p_{1}-2}\nabla \underline{%
			u}-|\nabla u|^{p_{1}-2}\nabla u)\nabla \zeta \ \mathrm{d}x+\int_{\{u<%
			\underline{u}\}}(|\underline{u}|^{p_{1}-2}\underline{u}-|u|^{p_{1}-2}u)\zeta
		\ \mathrm{d}x \\ 
		\leq \int_{\partial \Omega }|\nabla \underline{u}|^{p_{1}-2}\frac{\partial 
			\underline{u}}{\partial \eta }\gamma _{0}(\zeta )\text{ }\mathrm{d}s\leq 0,%
	\end{array}%
	\]%
	a contradiction. Hence $u\geq \underline{u}$ in $\Omega $. Arguing
	similarly, set $\hat{\zeta}=(u-\overline{u})^{+}$ and assume that $\hat{\zeta%
	}\neq 0.$ Then%
	\[
	\begin{array}{l}
		\int_{\{u>\overline{u}\}}|\nabla u|^{p_{1}-2}\nabla u\nabla \hat{\zeta}\ 
		\mathrm{d}x+\int_{\{u>\overline{u}\}}|u|^{p_{1}-2}u\hat{\zeta}\ \mathrm{d}x
		\\ 
		=\int_{\Omega }|\nabla u|^{p_{1}-2}\nabla u\nabla \hat{\zeta}\ \mathrm{d}%
		x+\int_{\Omega }|u|^{p_{1}-2}u\hat{\zeta}\ \mathrm{d}x=\int_{\Omega }f_{1}(x,%
		\mathrm{T}_{1}(u),\mathrm{T}_{2}(v))\hat{\zeta}\ \mathrm{d}x \\ 
		=\int_{\{u>\overline{u}\}}f_{1}(x,\mathrm{T}_{1}(u),\mathrm{T}_{2}(v))\hat{%
			\zeta}\ \mathrm{d}x=\int_{\{u>\overline{u}\}}f_{1}(x,\overline{u},\mathrm{T}%
		_{2}(v))\hat{\zeta}\ \mathrm{d}x \\ 
		\leq \int_{\{u>\overline{u}\}}|\nabla \overline{u}|^{p_{1}-2}\nabla 
		\overline{u}\nabla \hat{\zeta}\ \mathrm{d}x+\int_{\{u>\overline{u}\}}|%
		\overline{u}|^{p_{1}-2}\overline{u}\hat{\zeta}\ \mathrm{d}x-\int_{\partial
			\Omega }|\nabla \overline{u}|^{p_{1}-2}\frac{\partial \overline{u}}{\partial
			\eta }\gamma _{0}(\hat{\zeta})\text{ }\mathrm{d}s%
	\end{array}%
	\]%
	which leads to%
	\[
	\begin{array}{c}
		\int_{\{u>\overline{u}\}}(|\nabla u|^{p_{1}-2}\nabla u-|\nabla \overline{u}%
		|^{p_{1}-2}\nabla \overline{u})\nabla \hat{\zeta}\ \mathrm{d}x+\int_{\{u>%
			\overline{u}\}}(|u|^{p_{1}-2}u-|\overline{u}|^{p_{1}-2}\overline{u})\hat{%
			\zeta}\ \mathrm{d}x \\ 
		\leq -\int_{\partial \Omega }|\nabla \overline{u}|^{p_{1}-2}\frac{\partial 
			\overline{u}}{\partial \eta }\gamma _{0}(\hat{\zeta})\text{ }\mathrm{d}s\leq
		0,%
	\end{array}%
	\]%
	a contradiction. Thus, we have $u\leq \overline{u}$ in $\Omega $. A quite
	similar argument provides that $\underline{v}\leq v\leq \overline{v}$ in $\Omega $. Therefore, since $(\underline{u},\underline{v}),(\overline{u},
	\overline{v})\in L^{\infty }(\Omega )\times L^{\infty }(\Omega)$, we conclude that $(u,v)\in
	W_{b}^{1,p_{1}}(\Omega )\times W_{b}^{1,p_{2}}(\Omega )$. This completes the
	proof.{\footnotesize \ }
\end{proof}

Instead of $(\mathrm{H}_{2}),$ if we assume that $f_{1}$ and $f_{2}$ are
bounded in $\Omega \times \lbrack \underline{u},\overline{u}]\times \lbrack 
\underline{v},\overline{v}]$, the conclusion of Theorem \ref{T2} remains
true but assigning more regularity to the solution.

\begin{itemize}
	\item[$(\mathrm{H}_{3})$] There exists a constant $M>0$ such that%
	\begin{equation*}
		\left\vert f_{i}(x,u,v)\right\vert \leq M\text{ in }\Omega \times \lbrack 
		\underline{u},\overline{u}]\times \lbrack \underline{v},\overline{v}]\text{,
			for }i=1,2.
	\end{equation*}
\end{itemize}

\begin{theorem}
	\label{T4} Suppose $(\mathrm{H}_{1})$ and $(\mathrm{H}_{3})$ hold true.
	Then, problem $(\mathrm{P}_{f_{1},f_{2}})$ possesses a solution $(u,v)\in 
	\mathcal{C}^{1,\tau }(\overline{\Omega })\times \mathcal{C}^{1,\tau }(%
	\overline{\Omega })$ with suitable $\tau \in (0,1)$ such that 
	\begin{equation}
		\underline{u}\leq u\leq \overline{u}\quad \text{and}\quad \underline{v}\leq
		v\leq \overline{v}.  \label{19*}
	\end{equation}%
	Moreover, $\frac{\partial u}{\partial \eta }=\frac{\partial v}{\partial \eta 
	}=0$ on $\partial \Omega $.
\end{theorem}

\begin{proof}
	The proof is similar in the spirit to the one of Theorem \ref{T2}. Given $%
	(z_{1},z_{2})\in \mathcal{C}(\overline{\Omega })\times \mathcal{C}(\overline{%
		\Omega })$, we introduce the auxiliary system 
	\begin{equation}
		\left\{ 
		\begin{array}{ll}
			-\Delta _{p_{1}}u+|u|^{p_{1}-2}u=f_{1}(x,\mathrm{T}_{1}(z_{1}),\mathrm{T}%
			_{2}(z_{2})) & \text{in }\Omega , \\ 
			-\Delta _{p_{2}}v+|v|^{p_{2}-2}v=f_{2}(x,\mathrm{T}_{1}(z_{1}),\mathrm{T}%
			_{2}(z_{2})) & \text{in }\Omega , \\ 
			\frac{\partial u}{\partial \eta }=\frac{\partial v}{\partial \eta }=0 & 
			\text{on }\partial \Omega ,%
		\end{array}%
		\right.  \label{301*}
	\end{equation}%
	where the operators $\mathrm{T}_{1}$ and $\mathrm{T}_{2}$ are defined by (%
	\ref{0}). Notice that $(\mathrm{H}_{3})$ together with Minty-Browder Theorem
	(see, e.g., \cite[Theorem V.15]{B}) ensure that (\ref{301*}) admits a unique
	solution $(u,v)\in W^{1,p_{1}}(\Omega )\times W^{1,p_{2}}(\Omega ).$ Let us
	introduce the operator 
	\[
	\begin{array}{cccc}
		\mathcal{T}: & \mathcal{C}(\overline{\Omega })\times \mathcal{C}(\overline{%
			\Omega }) & \rightarrow & \mathcal{C}^{1}(\overline{\Omega })\times \mathcal{%
			C}^{1}(\overline{\Omega }) \\ 
		& (z_{1},z_{2}) & \mapsto & (u,v).%
	\end{array}%
	\]%
	Observe from (\ref{301*}) that any fixed point of $\mathcal{T}$ within $[%
	\underline{u},\overline{u}]\times \lbrack \underline{v},\overline{v}]$
	coincides with the weak solution of $(\mathrm{P}_{f_{1},f_{2}})$.
	
	By $(\mathrm{H}_{3})$ and according to the regularity result \cite{L}, it
	follows that $(u,v)\in \mathcal{C}^{1,\tau }(\overline{\Omega })\times 
	\mathcal{C}^{1,\tau }(\overline{\Omega })$ and 
	\begin{equation}
		\left\Vert u\right\Vert _{\mathcal{C}^{1,\tau }(\overline{\Omega }%
			)},\left\Vert v\right\Vert _{\mathcal{C}^{1,\tau }(\overline{\Omega })}\leq
		L_{0},  \label{309*}
	\end{equation}%
	for some constant $L_{0}>0$ independent of $u$ and $v$. Then, the
	compactness of the embedding $\mathcal{C}^{1,\tau }(\overline{\Omega }%
	)\subset \mathcal{C}^{1}(\overline{\Omega })$ implies that $\mathcal{T}(%
	\mathcal{C}(\overline{\Omega })\times \mathcal{C}(\overline{\Omega }))$ is a
	relatively compact subset of $\mathcal{C}^{1}(\overline{\Omega })\times 
	\mathcal{C}^{1}(\overline{\Omega })$. This proves that $\mathcal{T}$ is
	compact.
	
	Next, we show the continuity of $\mathcal{T}$ with respect to the topology
	of $\mathcal{C}(\overline{\Omega })\times \mathcal{C}(\overline{\Omega })$.
	Let $(z_{1,n},z_{2,n})\rightarrow (z_{1},z_{2})$ in $\mathcal{C}(\overline{%
		\Omega })\times \mathcal{C}(\overline{\Omega })$ for all $n$ and denote $%
	\left( u_{n},v_{n}\right) =\mathcal{T}(z_{1,n},z_{2,n})$. Repeating the
	previous argument in the proof of Theorem \ref{T2} we obtain 
	\[
	\begin{array}{c}
		(u_{n},v_{n})\rightarrow (u,v)\text{ in }W^{1,p_{1}}(\Omega )\times
		W^{1,p_{2}}(\Omega ),%
	\end{array}%
	\]%
	for certain $\left( u,v\right) $ in $W^{1,p_{1}}(\Omega )\times
	W^{1,p_{2}}(\Omega )$ and $\left( u,v\right) =\mathcal{T}(z_{1},z_{2})$. On
	the basis of \cite{L}, the sequence $\{\left( u_{n},v_{n}\right) \}$ is
	bounded in $\mathcal{C}^{1,\tau }(\overline{\Omega })\times \mathcal{C}%
	^{1,\tau }(\overline{\Omega })$ for certain $\tau \in (0,1)$. Then, through
	the compact embedding $\mathcal{C}^{1,\tau }(\overline{\Omega })\subset 
	\mathcal{C}^{1}(\overline{\Omega })$, along a relabeled subsequence, there
	holds $(u_{n},v_{n})\rightarrow (u,v)$ in $\mathcal{C}^{1}(\overline{\Omega }%
	)\times \mathcal{C}^{1}(\overline{\Omega })$, showing that $\mathcal{T}$ is
	continuous.
	
	We are thus in a position to apply Schauder's fixed point theorem to the map 
	$\mathcal{T}$, which establishes the existence of $(u,v)\in \mathcal{C}^{1}(%
	\overline{\Omega })\times \mathcal{C}^{1}(\overline{\Omega })$ satisfying $%
	(u,v)=\mathcal{T}(u,v).$ By $(\mathrm{H}_{3})$ together with \cite{L}, we
	infer that $(u,v)\in \mathcal{C}^{1,\tau }(\overline{\Omega })\times 
	\mathcal{C}^{1,\tau }(\overline{\Omega }),$ $\tau \in (0,1)$. The rest of
	the proof runs as the one of Theorem \ref{T2}.
\end{proof}

\begin{remark}
	Theorems \ref{T2} and \ref{T4} remain valid if we replace Neumann boundary
	conditions with Dirichlet ones.
\end{remark}

\section{Abstract results}

\label{S4}

\begin{theorem}
	\label{T7}For a constant $\Lambda >0$, let $(\underline{u}_{\Lambda },%
	\underline{v}_{\Lambda })$ and $(\overline{u}_{\Lambda },\overline{v}%
	_{\Lambda })$ be sub-supersolutions pairs of problem $(\mathrm{P})$. Assume (%
	\ref{c1}) holds and suppose there exists a constant $\rho >0$ such that 
	\begin{equation}
		\underline{u}_{\Lambda },\underline{v}_{\Lambda }>\rho \text{ a.e. in }%
		\overline{\Omega }.  \label{c4}
	\end{equation}
	Then, problem $(\mathrm{P})$ admits a solution $(u,v)\in int\mathcal{C}_{+}^{1,\tau }(\overline{\Omega })\times int\mathcal{C}_{+}^{1,\tau }(\overline{%
		\Omega }),$ for certain $\tau \in (0,1)$, within $[\underline{u}_{\Lambda },%
	\overline{u}_{\Lambda }]\times \lbrack \underline{v}_{\Lambda },\overline{v}%
	_{\Lambda }]$. Moreover, if $\alpha _{2},\beta _{1}\in (-1,0)$ and%
	\begin{equation}
		\max \{\frac{-\gamma _{1}\alpha _{i}}{p_{1}-1},\frac{-\gamma _{2}\beta _{i}}{%
			p_{2}-1}\}<p_{i}-1,\text{ }i=1,2,  \label{c10}
	\end{equation}%
	with%
	\begin{equation*}
		\gamma _{i}=\max \{-\alpha _{i},-\beta _{i}\},
	\end{equation*}%
	then, $(u,v)$ is unique.
\end{theorem}

\begin{proof}
	By (\ref{c1}), (\ref{c4}) and for all $(u,v)\in \lbrack \underline{u}%
	_{\Lambda },\overline{u}_{\Lambda }]\times \lbrack \underline{v}_{\Lambda },%
	\overline{v}_{\Lambda }],$ we have%
	\[
	u^{\alpha _{1}}+v^{\beta _{1}}\leq \left\{ 
	\begin{array}{ll}
		\underline{u}_{\Lambda }^{\alpha _{1}}+\overline{v}_{\Lambda }^{\beta _{1}}
		& \text{ if }\beta _{1}>0 \\ 
		\underline{u}_{\Lambda }^{\alpha _{1}}+\underline{v}_{\Lambda }^{\beta _{1}}
		& \text{ if }\beta _{1}<0%
	\end{array}%
	\right. \leq \left\{ 
	\begin{array}{ll}
		\rho ^{\alpha _{1}}+\left\Vert \overline{v}_{\Lambda }\right\Vert _{\infty
		}^{\beta _{1}} & \text{ if }\beta _{1}>0 \\ 
		\rho ^{\alpha _{1}}+\rho ^{\beta _{1}} & \text{ if }\beta _{1}<0,%
	\end{array}%
	\right.
	\]%
	as well as%
	\[
	u^{\alpha _{2}}+v^{\beta _{2}}\leq \left\{ 
	\begin{array}{ll}
		\overline{u}_{\Lambda }^{\alpha _{2}}+\underline{v}_{\Lambda }^{\beta _{2}}
		& \text{ if }\alpha _{2}>0 \\ 
		\underline{u}_{\Lambda }^{\alpha _{2}}+\underline{v}_{\Lambda }^{\beta _{2}}
		& \text{ if }\alpha _{2}<0%
	\end{array}%
	\right. \leq \left\{ 
	\begin{array}{ll}
		\left\Vert \overline{u}_{\Lambda }\right\Vert _{\infty }^{\alpha _{2}}+\rho
		^{\beta _{2}} & \text{ if }\alpha _{2}>0 \\ 
		\rho ^{\alpha _{2}}+\rho ^{\beta _{2}} & \text{ if }\alpha _{2}<0.%
	\end{array}%
	\right.
	\]%
	Then, thanks to Theorem \ref{T4}, we conclude that $(\mathrm{P})$ possesses
	a solution $(u,v)\in \mathcal{C}^{1,\tau }(\overline{\Omega })\times 
	\mathcal{C}^{1,\tau }(\overline{\Omega }),$ with suitable $\tau \in (0,1),$
	within $[\underline{u}_{\Lambda },\overline{u}_{\Lambda }]\times \lbrack 
	\underline{v}_{\Lambda },\overline{v}_{\Lambda }]$. Due to (\ref{c4}), we
	deduce that $(u,v)\in int\mathcal{C}_{+}^{1,\tau }(\overline{\Omega })\times int\mathcal{C}_{+}^{1,\tau }(\overline{\Omega })$.
	
	We proceed to show that $(u,v)$ is a unique solution in $[\underline{u}%
	_{\Lambda },\overline{u}_{\Lambda }]\times \lbrack \underline{v}_{\Lambda },%
	\overline{v}_{\Lambda }]$ for $\alpha _{2},\beta _{1}\in (-1,0)$. To this
	end, let $(u_{1},v_{1})$ and $(u_{2},v_{2})$ be two distinct positive
	solutions of $(\mathrm{P})$ within $[\underline{u}_{\Lambda },\overline{u}%
	_{\Lambda }]\times \lbrack \underline{v}_{\Lambda },\overline{v}_{\Lambda }]$%
	. Set%
	\[
	\tau =\sup \{c\in \mathbb{R}_{+}\text{, \ }cu_{2}\leq u_{1}\text{ \ and }%
	cv_{2}\leq v_{1}\text{ \ in\ }\Omega \}.
	\]%
	Then $0<\tau <\infty $ because 
	\begin{equation}
		\min \{\frac{\rho }{||u_{2}||_{\infty }},\frac{\rho }{||v_{2}||_{\infty }}%
		\}\leq \tau \leq \max \{\frac{||u_{1}||_{\infty }}{\rho },\frac{%
			||v_{1}||_{\infty }}{\rho }\}  \label{29}
	\end{equation}%
	If we manage to show that $\tau \geq 1$, we are done as this entails $%
	u_{2}\leq u_{1}$ and $v_{2}\leq v_{1}$ in $\Omega $ and thus, by
	interchanging the roles of $(u_{1},v_{1})$ and $(u_{2},v_{2})$ we get $%
	u_{2}\geq u_{1}$ and $v_{2}\geq v_{1}$ in $\Omega $. By contradiction,
	suppose that $0<\tau <1.$ Then, by (\ref{c1}) and since $\alpha _{2},\beta
	_{1}<0$, we infer that%
	\[
	\min \{\tau ^{-\alpha _{i}},\tau ^{-\beta _{i}}\}\geq \tau ^{\gamma _{i}}%
	\text{ \ and \ }\min \{\tau ^{-\frac{\gamma _{1}\alpha _{i}}{p_{1}-1}},\tau
	^{-\frac{\gamma _{2}\beta _{i}}{p_{2}-1}}\}\geq \tau ^{\hat{\gamma}_{i}}%
	\text{,}
	\]%
	with%
	\[
	\hat{\gamma}_{i}=\max \{\frac{-\gamma _{1}\alpha _{i}}{p_{1}-1},\frac{%
		-\gamma _{2}\beta _{i}}{p_{2}-1}\}>0,\text{for  }i=1,2.
	\]%
	A direct computations show that 
	\[
	\begin{array}{l}
		-\Delta _{p_{1}}u_{2}+|u_{2}|^{p_{1}-2}u_{2}=u_{2}^{\alpha
			_{1}}+v_{2}^{\beta _{1}}=\frac{(\tau u_{2})^{\alpha _{1}}}{\tau ^{\alpha
				_{1}}}+\frac{(\tau v_{2})^{\beta _{1}}}{\tau ^{\beta _{1}}} \\ 
		\geq \tau ^{\gamma _{1}}\left( (\tau u_{2})^{\alpha _{1}}+(\tau
		v_{2})^{\beta _{1}}\right) \geq \tau ^{\gamma _{1}}(u_{1}^{\alpha
			_{1}}+v_{1}^{\beta _{1}}) \\ 
		=\tau ^{\gamma _{1}}\left( -\Delta
		_{p_{1}}u_{1}+|u_{1}|^{p_{1}-2}u_{1}\right) \\ 
		=-\Delta _{p_{1}}(\tau ^{\frac{\gamma _{1}}{p_{1}-1}}u_{1})+|\tau ^{\frac{%
				\gamma _{1}}{p_{1}-1}}u_{1}|^{p_{1}-2}(\tau ^{\frac{\gamma _{1}}{p_{1}-1}%
		}u_{1})%
	\end{array}%
	\]%
	and similarly%
	\[
	\begin{array}{l}
		-\Delta _{p_{2}}v_{2}+|v_{2}|^{p_{2}-2}v_{2}=u_{2}^{\alpha
			_{2}}+v_{2}^{\beta _{2}}=\frac{(\tau u_{2})^{\alpha _{2}}}{\tau ^{\alpha
				_{2}}}+\frac{(\tau v_{2})^{\beta _{2}}}{\tau ^{\beta _{2}}} \\ 
		\geq \tau ^{\gamma _{2}}\left( (\tau u_{2})^{\alpha _{2}}+(\tau
		v_{2})^{\beta _{2}}\right) \geq \tau ^{\gamma _{2}}(u_{1}^{\alpha
			_{2}}+v_{1}^{\beta _{2}}) \\ 
		=-\Delta _{p_{2}}(\tau ^{\frac{\gamma _{2}}{p_{2}-1}}v_{1})+|\tau ^{\frac{%
				\gamma _{2}}{p_{2}-1}}v_{1}|^{p_{2}-2}(\tau ^{\frac{\gamma _{2}}{p_{2}-1}%
		}v_{1}).%
	\end{array}%
	\]%
	The weak comparison principle (see \cite[Lemma 3.2]{ST}) yields%
	\begin{equation}
		u_{2}\geq \tau ^{\frac{\gamma _{1}}{p_{1}-1}}u_{1}\text{ \ and \ }v_{2}\geq
		\tau ^{\frac{\gamma _{2}}{p_{2}-1}}v_{1}\text{ \ in }\Omega \text{.}
		\label{33}
	\end{equation}%
	Using (\ref{33}) in the equations for $(u_{1},v_{1})$, we get%
	\[
	\begin{array}{l}
		-\Delta _{p_{1}}u_{1}+|u_{1}|^{p_{1}-2}u_{1}=u_{1}^{\alpha
			_{1}}+v_{1}^{\beta _{1}}=\frac{(\tau ^{\frac{\gamma _{1}}{p_{1}-1}%
			}u_{1})^{\alpha _{1}}}{\tau ^{\frac{\gamma _{1}\alpha _{1}}{p_{1}-1}}}+\frac{%
			(\tau ^{\frac{\gamma _{2}}{p_{2}-1}}v_{1})^{\beta _{1}}}{\tau ^{\frac{\gamma
					_{2}\beta _{1}}{p_{2}-1}}} \\ 
		\geq \tau ^{\hat{\gamma}_{1}}(u_{2}^{\alpha _{1}}+v_{2}^{\beta _{1}})=\tau ^{%
			\hat{\gamma}_{1}}\left( -\Delta _{p_{1}}u_{2}+|u_{2}|^{p_{1}-2}u_{2}\right)
		\\ 
		=-\Delta _{p_{1}}(\tau ^{\frac{\hat{\gamma}_{1}}{p_{1}-1}}u_{2})+|\tau ^{%
			\frac{\hat{\gamma}_{1}}{p_{1}-1}}u_{2}|^{p_{1}-2}(\tau ^{\frac{\hat{\gamma}%
				_{1}}{p_{1}-1}}u_{2}).%
	\end{array}%
	\]%
	and 
	\[
	\begin{array}{l}
		-\Delta _{p_{2}}v_{1}+|v_{1}|^{p_{2}-2}v_{1}=u_{1}^{\alpha
			_{2}}+v_{1}^{\beta _{2}}=\frac{(\tau ^{\frac{\gamma _{1}}{p_{1}-1}%
			}u_{1})^{\alpha _{2}}}{\tau ^{\frac{\gamma _{1}\alpha _{2}}{p_{1}-1}}}+\frac{%
			(\tau ^{\frac{\gamma _{2}}{p_{2}-1}}v_{1})^{\beta _{2}}}{\tau ^{\frac{\gamma
					_{2}\beta _{2}}{p_{2}-1}}} \\ 
		\geq \tau ^{\hat{\gamma}_{2}}(u_{2}^{\alpha _{2}}+v_{2}^{\beta
			_{2}})=-\Delta _{p_{2}}(\tau ^{\frac{\hat{\gamma}_{2}}{p_{2}-1}}v_{2})+|\tau
		^{\frac{\hat{\gamma}_{2}}{p_{2}-1}}v_{2}|^{p_{2}-2}(\tau ^{\frac{\hat{\gamma}%
				_{2}}{p_{2}-1}}v_{2})\text{ in }\Omega .%
	\end{array}%
	\]%
	Owing to \cite[Lemma 3.2]{ST} we derive 
	\[
	u_{1}\geq \tau ^{\frac{\hat{\gamma}_{1}}{p_{1}-1}}u_{2}\text{ \ and \ }%
	v_{1}\geq \tau ^{\frac{\hat{\gamma}_{2}}{p_{2}-1}}v_{2}\text{ \ in }\Omega .
	\]%
	In view of (\ref{c10}), we deduce that 
	\[
	1>\tau ^{\frac{\hat{\gamma}_{1}}{p_{1}-1}}>\tau \text{ \ and \ }1>\tau ^{%
		\frac{\hat{\gamma}_{2}}{p_{2}-1}}>\tau ,
	\]%
	a contradiction with the definition of $\tau $. Thus, $\tau \geq 1$ and therefore, $%
	(u_{1},v_{1})=(u_{2},v_{2})$, completing the proof of the theorem.
\end{proof}

The existence result in the case of zero trace subsolutions on the boundary $%
\partial \Omega $ is formulated as follows.

\begin{theorem}
	\label{T8}For a constant $\Lambda >0$, let $(\underline{u}_{\Lambda },%
	\underline{v}_{\Lambda })$ and $(\overline{u}_{\Lambda },\overline{v}%
	_{\Lambda })$ be sub-supersolutions pairs of problem $(\mathrm{P})$ and
	assume that assumption (\ref{c1}) holds. If $(\underline{u}_{\Lambda },%
	\underline{v}_{\Lambda })\in W_{0,b}^{1,p_{1}}(\Omega )\times
	W_{0,b}^{1,p_{2}}(\Omega )$ and there exists a constant $c>0$ such that 
	\begin{equation}
		\underline{u}_{\Lambda },\underline{v}_{\Lambda }\geq cd(x)\text{ a.e. in }%
		\Omega,  \label{c6}
	\end{equation}%
	then, problem $(\mathrm{P})$ admits a solution $(u,v)\in
	W_{b}^{1,p_{1}}(\Omega )\times W_{b}^{1,p_{2}}(\Omega )$ within $[\underline{%
		u}_{\Lambda },\overline{u}_{\Lambda }]\times \lbrack \underline{v}_{\Lambda
	},\overline{v}_{\Lambda }]$.
\end{theorem}

\begin{proof}
	By (\ref{c1}) and (\ref{c6}), for all $(u,v)\in \lbrack \underline{u}%
	_{\Lambda },\overline{u}_{\Lambda }]\times \lbrack \underline{v}_{\Lambda },%
	\overline{v}_{\Lambda }],$ we have%
	\begin{eqnarray*}
		u^{\alpha _{1}}+v^{\beta _{1}} &\leq &\left\{ 
		\begin{array}{ll}
			\underline{u}_{\Lambda }^{\alpha _{1}}+\overline{v}_{\Lambda }^{\beta _{1}}
			& \text{ if }\beta _{1}>0 \\ 
			\underline{u}_{\Lambda }^{\alpha _{1}}+\underline{v}_{\Lambda }^{\beta _{1}}
			& \text{ if }\beta _{1}<0%
		\end{array}%
		\right. \\
		&\leq &\left\{ 
		\begin{array}{ll}
			(cd(x))^{\alpha _{1}}+\left\Vert \overline{v}_{\Lambda }\right\Vert _{\infty
			}^{\beta _{1}} & \text{ if }\beta _{1}>0 \\ 
			(cd(x))^{\alpha _{1}}+(cd(x))^{\beta _{1}} & \text{ if }\beta _{1}<0%
		\end{array}%
		\right. \\
		&\leq &\left\{ 
		\begin{array}{ll}
			(c^{\alpha _{1}}+d(x)^{-\alpha _{1}}\left\Vert \overline{v}_{\Lambda
			}\right\Vert _{\infty }^{\beta _{1}})d(x)^{\alpha _{1}} & \text{ if }\beta
			_{1}>0 \\ 
			(c^{\alpha _{1}}+c^{\beta _{1}})\max \{d(x))^{\alpha _{1}},d(x))^{\beta
				_{1}}\} & \text{ if }\beta _{1}<0%
		\end{array}%
		\right. \\
		&\leq &C_{0}\left\{ 
		\begin{array}{ll}
			d(x)^{\alpha _{1}} & \text{ if }\beta _{1}>0 \\ 
			\max \{d(x)^{\alpha _{1}},d(x)^{\beta _{1}}\} & \text{ if }\beta _{1}<0%
		\end{array}%
		\right. \text{, in }\Omega
	\end{eqnarray*}%
	and similarly
	\begin{eqnarray*}
		u^{\alpha _{2}}+v^{\beta _{2}} &\leq &\left\{ 
		\begin{array}{ll}
			\overline{u}_{\Lambda }^{\alpha _{2}}+\underline{v}_{\Lambda }^{\beta _{2}}
			& \text{ if }\alpha _{2}>0 \\ 
			\underline{u}_{\Lambda }^{\alpha _{2}}+\underline{v}_{\Lambda }^{\beta _{2}}
			& \text{ if }\alpha _{2}<0%
		\end{array}%
		\right. \leq \left\{ 
		\begin{array}{ll}
			\left\Vert \overline{u}_{\Lambda }\right\Vert _{\infty }^{\alpha
				_{2}}+(cd(x))^{\beta _{2}} & \text{ if }\alpha _{2}>0 \\ 
			(cd(x))^{\alpha _{2}}+(cd(x))^{\beta _{2}} & \text{ if }\alpha _{2}<0.%
		\end{array}%
		\right. \\
		&\leq &\tilde{C}_{0}\left\{ 
		\begin{array}{ll}
			d(x)^{\beta _{2}} & \text{ if }\alpha _{2}>0 \\ 
			\max \{d(x)^{\alpha _{2}},d(x)^{\beta _{2}}\} & \text{ if }\alpha _{2}<0%
		\end{array}%
		\right. \text{, in }\Omega ,
	\end{eqnarray*}%
	for certain constants $C_{0},\tilde{C}_{0}>0$. Then, thanks to Theorem \ref%
	{T2}, we conclude that there exists a solution $(u,v)\in
	W_{b}^{1,p_{1}}(\Omega )\times W_{b}^{1,p_{2}}(\Omega )$ of problem $(%
	\mathrm{P})$ within $[\underline{u}_{\Lambda },\overline{u}_{\Lambda
	}]\times \lbrack \underline{v}_{\Lambda },\overline{v}_{\Lambda }]$. This
	proves the theorem.
\end{proof}

\begin{remark}
	\label{R1}It is worth noting that the supersolution $(\overline{u}_{\Lambda
	},\overline{v}_{\Lambda })$ in Theorem \ref{T8} should not belong to $%
	W_{0}^{1,p_{1}}(\Omega )\times W_{0}^{1,p_{2}}(\Omega )$ because if so, the
	solution $(u,v)$ would be of zero trace on the boundary $\partial \Omega .$
	Hence, \cite[Lemma 3.1]{H} ensures that $(u,v)\in \mathcal{C}^{1,\tau }(%
	\overline{\Omega })\times \mathcal{C}^{1,\tau }(\overline{\Omega })$ which,
	by Hopf's Lemma (see, e.g., \cite{AC}), implies that $\frac{\partial 
		\overline{u}}{\partial \eta },\frac{\partial \overline{v}}{\partial \eta }<0$
	on $\partial \Omega $, absurd.
\end{remark}

\begin{remark}
	The argument used to show the uniqueness result of Theorem \ref{T7} is not
	applicable in the context of Theorem \ref{T8}. Being impossible to construct supersolutions for problem $(\mathrm{P})$ behaving like the distance function $d(x)$ (see Remark \ref{R1}), then it would not be possible to get	an estimate of type (\ref{29}).
\end{remark}

\section{Existence and uniqueness results}

\label{S2}

Our goal is to construct sub- and super-solution pairs of $(\mathrm{P})$.
With this aim, consider the following nonlinear Dirichlet and Neumann
eigenvalue problems%
\begin{equation}
	\left\{ 
	\begin{array}{l}
		-\Delta _{p_{i}}\phi _{1,p_{i}}+|\phi _{1,p_{i}}|^{p_{i}-2}\phi
		_{1,p_{i}}=\lambda _{1,p_{i}}|\phi _{1,p_{i}}|^{p_{i}-2}\phi _{1,p_{i}}\text{
			in }\Omega \\ 
		\phi _{1,p_{i}}=0\text{ on }\partial \Omega ,\text{ }i=1,2,%
	\end{array}%
	\right.  \label{3}
\end{equation}%
\begin{equation}
	\left\{ 
	\begin{array}{l}
		-\Delta _{p_{i}}\hat{\phi}_{1,p_{i}}+|\hat{\phi}_{1,p_{i}}|^{p_{i}-2}\hat{%
			\phi}_{1,p_{i}}=\hat{\lambda}_{1,p_{i}}|\hat{\phi}_{1,p_{i}}|^{p_{i}-2}\hat{%
			\phi}_{1,p_{i}}\text{ in }\Omega \\ 
		\frac{\partial \hat{\phi}_{1,p_{i}}}{\partial \eta }=0\text{ on }\partial
		\Omega ,\text{ }i=1,2,%
	\end{array}%
	\right.  \label{3*}
\end{equation}%
where $\phi _{1,p_{i}}\in \mathcal{C}_{+}^{1}(\overline{\Omega })$ and $\hat{%
	\phi}_{1,p_{i}}\in int\mathcal{C}_{+}^{1}(\overline{\Omega })$ are the
eigenfunctions corresponding to the first eigenvalues $\lambda _{1,p_{i}},%
\hat{\lambda}_{1,p_{i}}>0,$ respectively (see \cite{MMP}).
Recall that%
\begin{equation}
	\phi _{1,p_{i}}\geq c_{0}d(x)\text{ in }\Omega ,\text{ \ and \ }\frac{%
		\partial \phi _{1,p_{i}}}{\partial \eta }<0\;\text{on}\;\partial \Omega ,
	\label{5}
\end{equation}%
for a constant $c_{0}>0,$ and there exists a constant $\mu >0$ such that%
\begin{equation}
	\hat{\phi}_{1,p_{i}}(x)>\mu \text{ for all }x\in \overline{\Omega }.
	\label{5*}
\end{equation}%
Consider the homogeneous Dirichlet and Neumann problems%
\begin{equation}
	\left\{ 
	\begin{array}{l}
		-\Delta _{p_{i}}y_{i}+|y_{i}|^{p_{i}-2}y_{i}=1\text{ in }\Omega , \\ 
		y_{i}=0\text{ on }\partial \Omega ,\text{ }i=1,2%
	\end{array}%
	\right.  \label{6}
\end{equation}%
\begin{equation}
	\left\{ 
	\begin{array}{l}
		-\Delta _{p_{i}}\hat{y}_{i}+|\hat{y}_{i}|^{p_{i}-2}\hat{y}_{i}=1\text{ in }%
		\Omega \\ 
		\frac{\partial \hat{y}_{i}}{\partial \eta }=0\text{ on }\partial \Omega ,%
		\text{ }i=1,2,%
	\end{array}%
	\right.  \label{6*}
\end{equation}%
which admit unique solutions $y_{i}\in \mathcal{C}^{1,\tau }(\overline{%
	\Omega })$ and $\hat{y}_{i}\in int\mathcal{C}_{+}^{1}(\overline{\Omega })$
satisfying 
\begin{equation}
	\frac{d}{c}\leq y_{i}\leq cd\;\text{in}\;\Omega ,\;\;\frac{\partial y_{i}}{%
		\partial \eta }<0\;\text{on}\;\partial \Omega ,  \label{4}
\end{equation}%
\begin{equation}
	\frac{\hat{\phi}_{1,p_{i}}}{\hat{c}}\leq \hat{y}_{i}\leq \hat{c}\hat{\phi}%
	_{1,p_{i}}\;\text{in}\;\Omega ,  \label{4*}
\end{equation}%
for some constants $c,\hat{c}>1$ and $L>0$ (see \cite{DM, ST, V} ).

\subsection{$C^{1}$- bound solutions}

\begin{theorem}
	\label{T1} Assume (\ref{c1}) is satisfied. Then, if $\Lambda >0$ is big
	enough, problem $(\mathrm{P})$ admits a solution $(u,v)\in int \mathcal{C}_{+}^{1,\tau }(\overline{\Omega })\times int\mathcal{C}_{+}^{1,\tau }(\overline{\Omega }%
	),$ for certain $\tau \in (0,1)$, such that%
	\begin{equation}
		(u,v)\in \lbrack \Lambda ^{-1}\hat{\phi}_{1,p_{1}},\Lambda \hat{y}%
		_{1}]\times \lbrack \Lambda ^{-1}\hat{\phi}_{1,p_{2}},\Lambda \hat{y}_{2}].
		\label{27}
	\end{equation}%
	which is unique once $\alpha _{2},\beta _{1}\in (-1,0)$ and assumption (\ref%
	{c10}) is fulfilled .
\end{theorem}

\begin{proof}
	For a constant $\Lambda >0$ which will be specified later, let us show that $%
	\Lambda (\hat{y}_{1},\hat{y}_{2})$ satisfy (\ref{c3}). With this aim, pick $%
	(u,v)\in W^{1,p_{1}}(\Omega )\times W^{1,p_{2}}(\Omega )$ within $[\Lambda
	^{-1}\hat{\phi}_{1,p_{1}},\Lambda \hat{y}_{1}]\times \lbrack \Lambda ^{-1}%
	\hat{\phi}_{1,p_{2}},\Lambda \hat{y}_{2}]$. From (\ref{4*}) and (\ref{c1}),
	if $\beta _{1},\alpha _{2}>0$, we get%
	\begin{equation}
		\begin{array}{l}
			(\Lambda \hat{y}_{1})^{\alpha _{1}}+(\Lambda \hat{y}_{2})^{\beta _{1}}\leq
			(\Lambda \frac{\hat{\phi}_{1,p_{1}}}{\hat{c}})^{\alpha _{1}}+(\Lambda \hat{c}%
			\hat{\phi}_{1,p_{2}})^{\beta _{1}} \\ 
			\leq \hat{c}^{-\alpha _{1}}+(\Lambda \hat{c}||\hat{\phi}_{1,p_{2}}||_{\infty
			})^{\beta _{1}}\leq \Lambda ^{\beta _{1}}(1+(\hat{c}||\hat{\phi}%
			_{1,p_{2}}||_{\infty })^{\beta _{1}})\text{ \ in }\Omega ,%
		\end{array}
		\label{11}
	\end{equation}%
	\begin{equation}
		\begin{array}{l}
			(\Lambda \hat{y}_{1})^{\alpha _{2}}+(\Lambda \hat{y}_{2})^{\beta _{2}}\leq
			(\Lambda \hat{c}\hat{\phi}_{1,p_{1}})^{\alpha _{2}}+(\Lambda \frac{\hat{\phi}%
				_{1,p_{2}}}{\hat{c}})^{\beta _{2}} \\ 
			\leq (\Lambda \hat{c}||\hat{\phi}_{1,p_{1}}||_{\infty })^{\alpha _{2}}+\hat{c%
			}^{-\beta _{2}}\leq \Lambda ^{\alpha _{2}}(1+(\hat{c}||\hat{\phi}%
			_{1,p_{1}}||_{\infty })^{\alpha _{2}})\text{ \ in }\Omega ,%
		\end{array}%
	\end{equation}%
	while if $\beta _{1},\alpha _{2}<0$, we obtain%
	\begin{equation}
		\begin{array}{l}
			(\Lambda \hat{y}_{1})^{\alpha _{1}}+(\Lambda ^{-1}\hat{\phi}%
			_{1,p_{2}})^{\beta _{1}}\leq (\Lambda \frac{\hat{\phi}_{1,p_{1}}}{\hat{c}}%
			)^{\alpha _{1}}+(\Lambda ^{-1}\mu )^{\beta _{1}} \\ 
			\leq \hat{c}^{-\alpha _{1}}+(\Lambda ^{-1}\mu )^{\beta _{1}}\leq \Lambda
			^{-\beta _{1}}(1+\mu ^{\beta _{1}})\text{ \ in }\Omega ,%
		\end{array}%
	\end{equation}%
	\begin{equation}
		\begin{array}{l}
			(\Lambda ^{-1}\hat{\phi}_{1,p_{1}})^{\alpha _{2}}+(\Lambda \hat{y}%
			_{2})^{\beta _{2}}\leq (\Lambda ^{-1}\mu )^{\alpha _{2}}+(\Lambda \frac{\hat{%
					\phi}_{1,p_{2}}}{\hat{c}})^{\beta _{2}} \\ 
			\leq (\Lambda ^{-1}\mu )^{\alpha _{2}}+\hat{c}^{-\beta _{2}}\leq \Lambda
			^{-\alpha _{2}}(\mu ^{\alpha _{2}}+1)\text{ \ in }\Omega ,%
		\end{array}%
	\end{equation}%
	provided $\Lambda >0$ is sufficiently large. By (\ref{6*}) one has%
	\begin{equation}
		\begin{array}{l}
			-\Delta _{p_{i}}(\Lambda \hat{y}_{i})+|\Lambda \hat{y}_{i}|^{p_{i}-2}(%
			\Lambda \hat{y}_{i})=\Lambda ^{p_{i}-1}\text{\ in }\Omega ,\text{ for }i=1,2.%
		\end{array}
		\label{12}
	\end{equation}%
	Then, in view of (\ref{c1}), gathering (\ref{11})-(\ref{12}) together imply%
	\begin{eqnarray}
		&&-\Delta _{p_{1}}(\Lambda \hat{y}_{1})+|\Lambda \hat{y}_{1}|^{p_{1}-2}(%
		\Lambda \hat{y}_{1})  \label{17} \\
		&\geq &\left\{ 
		\begin{array}{ll}
			(\Lambda \hat{y}_{1})^{\alpha _{1}}+(\Lambda \hat{y}_{2})^{\beta _{1}} & 
			\text{if }\beta _{1}>0 \\ 
			(\Lambda \hat{y}_{1})^{\alpha _{1}}+(\Lambda ^{-1}\hat{\phi}%
			_{1,p_{2}})^{\beta _{1}} & \text{if }\beta _{1}<0%
		\end{array}%
		\right. \\
		&\geq &(\Lambda \hat{y}_{1})^{\alpha _{1}}+v^{\beta _{1}}\text{ \ in }\Omega
		\nonumber
	\end{eqnarray}%
	and%
	\begin{eqnarray}
		&&-\Delta _{p_{2}}(\Lambda \hat{y}_{2})+|\Lambda \hat{y}_{2}|^{p_{2}-2}(%
		\Lambda \hat{y}_{2})  \label{17*} \\
		&\geq &\left\{ 
		\begin{array}{ll}
			(\Lambda \hat{y}_{1})^{\alpha _{2}}+(\Lambda \hat{y}_{2})^{\beta _{2}} & 
			\text{if }\alpha _{2}>0 \\ 
			(\Lambda ^{-1}\hat{\phi}_{1,p_{1}})^{\alpha _{2}}+(\Lambda \hat{y}%
			_{2})^{\beta _{2}} & \text{if }\alpha _{2}<0%
		\end{array}%
		\right. \\
		&\geq &u^{\alpha _{1}}+(\Lambda \hat{y}_{2})^{\beta _{2}}\text{ \ in }\Omega
		,  \nonumber
	\end{eqnarray}%
	for $\Lambda >1$ large enough. According to (\ref{6*}), (\ref{c1*}) is
	fulfilled for $(\overline{u},\overline{v}):=\Lambda (\hat{y}_{1},\hat{y}%
	_{2}) $ and therefore, in view of (\ref{17})-(\ref{17*}), (\ref{c3}) too.
	
	We claim that (\ref{c2}) holds for $(\underline{u},\underline{v}):=\Lambda
	^{-1}(\hat{\phi}_{1,p_{1}},\hat{\phi}_{1,p_{2}})$. From (\ref{3*}), (\ref{4*}) and (\ref{c1}), we have%
	\begin{equation}
		(\Lambda ^{-1}\hat{\phi}_{1,p_{1}})^{\alpha _{1}}+v^{\beta _{1}}\geq
		(\Lambda ^{-1}\hat{\phi}_{1,p_{1}})^{\alpha _{1}}\geq (\Lambda ^{-1}||\hat{%
			\phi}_{1,p_{1}}||_{\infty })^{\alpha _{1}}\text{\ \ in }\Omega ,  \label{26}
	\end{equation}%
	\begin{equation}
		u^{\alpha _{2}}+(\Lambda ^{-1}\hat{\phi}_{1,p_{2}})^{\beta _{2}}\geq
		(\Lambda ^{-1}\hat{\phi}_{1,p_{2}})^{\beta _{2}}\geq (\Lambda ^{-1}||\hat{\phi}_{1,p_{2}}||_{\infty })^{\beta _{2}}\text{\ \ in }\Omega ,
	\end{equation}%
	as well as
	\begin{equation}
		\begin{array}{l}
			-\Delta _{p_{i}}(\Lambda ^{-1}\hat{\phi}_{1,p_{i}})+|\Lambda ^{-1}\hat{\phi}%
			_{1,p_{i}}|^{p_{i}-2}(\Lambda ^{-1}\hat{\phi}_{1,p_{i}})=\hat{\lambda}%
			_{1,p_{i}}(\Lambda ^{-1}\hat{\phi}_{1,p_{i}})^{p_{i}-1} \\ 
			\leq \hat{\lambda}_{1,p_{i}}(\Lambda ^{-1}||\hat{\phi}_{1,p_{1}}||_{\infty
			})^{p_{i}-1}\text{\ in }\Omega ,\text{ for }i=1,2.%
		\end{array}
		\label{26*}
	\end{equation}%
	Then, gathering (\ref{26})-(\ref{26*}) together, since (\ref{c1*}) is
	fulfilled (in view of (\ref{3*})), we deduce that (\ref{c2}) is achieved for $%
	\Lambda >0$ large enough. Consequently, considering that functions $\Lambda
	^{-1}\hat{\phi}_{1,p_{1}}$ and $\Lambda ^{-1}\hat{\phi}_{1,p_{2}}$ fulfill (%
	\ref{c4}), Theorem \ref{T7} ensures the existence of a positive solution $%
	(u,v)\in int \mathcal{C}_{+}^{1,\tau }(\overline{\Omega })\times int\mathcal{C}_{+}^{1,\tau }(\overline{\Omega }%
	),$ verifying (\ref{27}). Moreover, $(u,v)$ is unique
	once $\alpha _{2},\beta _{1}\in (-1,0)$ and fulfill assumption (\ref{c10})
	in Theorem \ref{T7}. This ends the proof.
\end{proof}

By making some specific and necessary adjustments, it is quite possible to
construct a pair of sub-supersolution in the spirit of \cite{GM}. However,
it should be noted that the sub-supersolution produced in \cite{GM} does not
fulfill the required assumptions (\ref{c2})-(\ref{c3}), in particular the
condition (\ref{c1*}) which in no case could be ignored.

Let $\Lambda >0$ be a constant so large such that%
\begin{equation}
	\Lambda >\max_{i=1,2}\{2(1+\frac{3}{( 3^{p_{i}-1} - 2^{p_{i}-1})^{\frac{1}{p_{i}-1}}}+||\hat{\phi}_{1,p_{i}}||_{\infty } +\Vert
	y_{i}\Vert _{\infty } )\} .  \label{2}
\end{equation}%
Define%
\begin{equation}
	\underline{u}_{\Lambda }:=\Lambda ^{-1}(\Lambda -\hat{\phi}_{1,p_{1}}),\quad 
	\underline{v}_{\Lambda }:=\Lambda ^{-1}(\Lambda -\hat{\phi}_{1,p_{2}}),
	\label{13*}
\end{equation}%
\begin{equation}
	\overline{u}_{\Lambda }:=\Lambda (\Lambda -y_{1}),\quad \overline{v}%
	_{\Lambda }:=\Lambda (\Lambda -y_{2}).  \label{13}
\end{equation}%
Obviously, $\underline{u}_{\Lambda }\leq \overline{u}_{\Lambda }$ and $%
\underline{v}_{\Lambda }\leq \overline{v}_{\Lambda }$. Moreover, via (\ref{6}) and (\ref{3*}) we have 
\begin{equation}
	\frac{\partial \underline{u}_{\Lambda }}{\partial \eta }=-\Lambda \frac{%
		\partial \hat{\phi}_{1,p_{1}}}{\partial \eta }=0,\text{ \ }\frac{\partial 
		\underline{v}_{\Lambda }}{\partial \eta }=-\Lambda \frac{\partial \hat{\phi}%
		_{1,p_{1}}}{\partial \eta }=0\text{ \ on \ }\partial \Omega  \label{14}
\end{equation}%
and 
\begin{equation}
	\frac{\partial \overline{u}_{\Lambda }}{\partial \eta }=-\Lambda \frac{%
		\partial y_{1}}{\partial \eta }>0,\text{ \ }\ \frac{\partial \overline{v}%
		_{\Lambda }}{\partial \eta }=-\Lambda \frac{\partial y_{2}}{\partial \eta }>0%
	\text{ \ \ on \ }\partial \Omega ,  \label{14*}
\end{equation}%
showing that $(\underline{u}_{\Lambda },\underline{v}_{\Lambda })$ and $(%
\overline{u}_{\Lambda },\overline{v}_{\Lambda })$ fulfill assumption (\ref%
{c1*}).

\begin{theorem}
	\label{T3} Assume (\ref{c1}) holds. Then, for $\Lambda >0$ is big enough,
	problem $(\mathrm{P})$ admits a solution $(u,v)\in int \mathcal{C}_{+}^{1,\tau }(\overline{\Omega })\times int\mathcal{C}_{+}^{1,\tau }(\overline{\Omega }%
	),$ $\tau \in (0,1)$, within $[\underline{u}_{\Lambda },\overline{u}_{\Lambda
	}]\times \lbrack \underline{v}_{\Lambda },\overline{u}_{\Lambda }]$, which
	is unique once $\alpha _{2},\beta _{1}\in (-1,0) $ and assumption (\ref{c10}%
	) is fulfilled .
\end{theorem}

\begin{proof}
	Let $(u,v)\in W^{1,p_{1}}(\Omega )\times W^{1,p_{2}}(\Omega )$ within $[%
	\underline{u}_{\Lambda },\overline{u}_{\Lambda }]\times \lbrack \underline{v}%
	_{\Lambda },\overline{v}_{\Lambda }]$. From (\ref{2}) observe that
	\begin{equation}
		\Lambda -||\hat{\phi}_{1,p_{i}}||_{\infty } > \frac{\Lambda }{2} \text{ and }
		\Lambda - \Vert y_{i}\Vert _{\infty } > \frac{\Lambda }{2} ,\text{ for }%
		i=1,2.  \label{7}
	\end{equation}
	Moreover, on the assumption that 
	\begin{equation}
		\Lambda > \frac{6}{( 3^{p_{i}-1} - 2^{p_{i}-1})^{\frac{1}
				{p_{i}-1}}}
	\end{equation}
	we derive that
	\begin{equation}
		( \frac{\Lambda}{2}) ^{p_{i}-1} - ( \frac{\Lambda}{3} )^{p_{i}-1}>1
		,\text{ for } i=1,2. \label{77}
	\end{equation}
	By (\ref{7}), (\ref{13}) and (\ref{77}), a direct computations give%
	\begin{equation}
		\begin{array}{l}
			-\Delta _{p_{1}}\overline{u}_{\Lambda }+|\overline{u}_{\Lambda }|^{p_{1}-2}%
			\overline{u}_{\Lambda }=\Delta _{p_{1}}(\Lambda y_{1})+\overline{u}_{\Lambda
			}^{p_{1}-1} \\ 
			=\Lambda ^{p_{1}-1}(-1+y_{1}^{p_{1}-1}+(\Lambda -y_{1})^{p_{1}-1}) \\ 
			\geq \Lambda ^{p_{1}-1}(-1+(\Lambda -\left\Vert y_{1}\right\Vert _{\infty
			})^{p_{1}-1}) \\ 
			\geq \Lambda ^{p_{1}-1}(-1+(\frac{\Lambda}{2}) ^{%
				p_{1}-1})   \\ 
			\geq \Lambda ^{p_{1}-1}  (\frac{\Lambda}{3})^{p_{1}-1} 
			\geq  (\frac{\Lambda^{2}}{3})^{p_{1}-1}  \text{ in }\Omega.
		\end{array}
		\label{8}
	\end{equation}%
	If $\beta _{1}>0$, it follows from (\ref{7}), (\ref{13}) and (\ref{c1}) that 
	\begin{equation}
		\begin{array}{l}
			\overline{u}_{\Lambda }^{\alpha _{1}}+v^{\beta _{1}}\leq \overline{u}%
			_{\Lambda }^{\alpha _{1}}+\overline{v}_{\Lambda }^{\beta _{1}}=\Lambda
			^{\alpha _{1}}(\Lambda -y_{1})^{\alpha _{1}}+\Lambda ^{\beta _{1}}(\Lambda
			-y_{2})^{\beta _{1}} \\ 
			\leq \Lambda ^{\alpha _{1}}(\Lambda -\left\Vert y_{1}\right\Vert _{\infty
			})^{\alpha _{1}}+\Lambda ^{2\beta _{1}}\leq 1+\Lambda ^{2\beta _{1}}\text{
				in }\Omega ,%
		\end{array}%
	\end{equation}%
	while, if $\beta _{1}<0$, we get 
	\begin{equation}
		\begin{array}{l}
			\overline{u}_{\Lambda }^{\alpha _{1}}+v^{\beta _{1}}\leq \overline{u}%
			_{\Lambda }^{\alpha _{1}}+\underline{v}_{\Lambda }^{\beta _{1}}=\Lambda
			^{\alpha _{1}}(\Lambda -y_{1})^{\alpha _{1}}+\Lambda ^{-\beta _{1}}(\Lambda -%
			\hat{\phi}_{1,p_{2}})^{\beta _{1}} \\ 
			\leq \Lambda ^{\alpha _{1}}(\Lambda -\left\Vert y_{1}\right\Vert _{\infty
			})^{\alpha _{1}}+\Lambda ^{-\beta _{1}}(\Lambda -||\hat{\phi}%
			_{1,p_{2}}||_{\infty })^{\beta _{1}} \\ 
			\leq 1+2^{-\beta _{1}}\Lambda ^{-\beta _{1}}\Lambda ^{\beta
				_{1}}=1+2^{-\beta _{1}}\text{ in }\Omega .%
		\end{array}
		\label{8*}
	\end{equation}%
	Then, for $\Lambda >0$ large enough, (\ref{8})-(\ref{8*}) result in%
	\[
	-\Delta _{p_{1}}\overline{u}_{\Lambda }+|\overline{u}_{\Lambda }|^{p_{1}-2}%
	\overline{u}_{\Lambda }\geq \overline{u}_{\Lambda }^{\alpha _{1}}+v^{\beta
		_{1}}\text{ in }\Omega ,
	\]%
	for all $v\in \lbrack \underline{v}_{\Lambda },\overline{v}_{\Lambda }]$.
	Similar arguments apply to the second equation in $(\mathrm{P})$ lead to%
	\[
	-\Delta _{p_{2}}\overline{v}_{\Lambda }+|\overline{v}_{\Lambda }|^{p_{2}-2}%
	\overline{v}_{\Lambda }\geq u^{\alpha _{2}}+\overline{v}_{\Lambda }^{\beta
		_{1}}\text{ in }\Omega ,
	\]%
	for all $u\in \lbrack \underline{u}_{\Lambda },\overline{u}_{\Lambda }]$.
	Bearing in mind (\ref{14*}), it follows that $(\overline{u}_{\Lambda },%
	\overline{v}_{\Lambda })$ in (\ref{13}) verifies assumption (\ref{c3}).
	Consequently, $(\overline{u}_{\Lambda },\overline{v}_{\Lambda })$ is a
	supersolution of $(\mathrm{P})$.
	
	Next, we prove that $(\underline{u}_{\Lambda },\underline{v}_{\Lambda })$ in
	(\ref{13*}) is a subsolution of $(\mathrm{P}).$ It is important to mention
	that the eigenvalues in (\ref{3*}) verify $\hat{\lambda}_{1,p_{1}},\hat{%
		\lambda}_{1,p_{2}}>1$. Hence, using (\ref{c1}), (\ref{3*}) and (\ref{7}), we
	obtain%
	\begin{equation}
		\begin{array}{l}
			-\Delta _{p_{1}}\underline{u}_{\Lambda }+|\underline{u}_{\Lambda }|^{p_{1}-2}%
			\underline{u}_{\Lambda }=\Delta _{p_{1}}(\Lambda ^{-1}\hat{\phi}_{1,p_{1}})+%
			\underline{u}_{\Lambda }^{p_{1}-1} \\ 
			=\Lambda ^{-(p_{1}-1)}(1-\hat{\lambda}_{1,p_{1}})\hat{\phi}%
			_{1,p_{1}}^{p_{1}-1}+\Lambda ^{-(p_{1}-1)}(\Lambda -\hat{\phi}%
			_{1,p_{1}})^{p_{1}-1} \\ 
			\leq \Lambda ^{-(p_{1}-1)}(\Lambda -\hat{\phi}_{1,p_{1}})^{p_{1}-1}\leq
			\Lambda ^{-(p_{1}-1)}(\Lambda -||\hat{\phi}_{1,p_{1}}||_{\infty })^{p_{1}-1}
			\\ 
			\leq (\frac{\Lambda -||\hat{\phi}_{1,p_{1}}||_{\infty }}{\Lambda }%
			)^{p_{1}-1}\leq 1\text{ in }\Omega%
		\end{array}
		\label{37}
	\end{equation}%
	and%
	\begin{equation}
		\begin{array}{l}
			-\Delta _{p_{2}}\underline{v}_{\Lambda }+|\underline{v}_{\Lambda }|^{p_{2}-2}%
			\underline{v}_{\Lambda }=\Delta _{p_{2}}(\Lambda ^{-1}\hat{\phi}_{1,p_{2}})+%
			\underline{v}_{\Lambda }^{p_{2}-1} \\ 
			=\Lambda ^{-(p_{2}-1)}(1-\hat{\lambda}_{1,p_{2}})\hat{\phi}%
			_{1,p_{2}}^{p_{2}-1}+\Lambda ^{-(p_{2}-1)}(\Lambda -\hat{\phi}%
			_{1,p_{2}})^{p_{2}-1} \\ 
			\leq \Lambda ^{-(p_{2}-1)}(\Lambda -\hat{\phi}_{1,p_{2}})^{p_{2}-1}\leq
			\Lambda ^{-(p_{2}-1)}(\Lambda -||\hat{\phi}_{1,p_{2}}||_{\infty })^{p_{2}-1}
			\\ 
			\leq (\frac{\Lambda -||\hat{\phi}_{1,p_{2}}||_{\infty }}{\Lambda }%
			)^{p_{2}-1}\leq 1\text{ in }\Omega%
		\end{array}%
	\end{equation}%
	as well as 
	\begin{equation}
		\underline{u}_{\Lambda }^{\alpha _{1}}+v^{\beta _{1}}\geq \underline{u}%
		_{\Lambda }^{\alpha _{1}}=\Lambda ^{-\alpha _{1}}(\Lambda -\hat{\phi}%
		_{1,p_{1}})^{\alpha _{1}}\geq \Lambda ^{-\alpha _{1}}\Lambda ^{\alpha _{1}}=1%
		\text{ in }\Omega ,
	\end{equation}%
	and%
	\begin{equation}
		u_{\Lambda }^{\alpha _{2}}+\underline{v}_{\Lambda }^{\beta _{2}}\geq 
		\underline{v}_{\Lambda }^{\beta _{2}}=\Lambda ^{-\beta _{2}}(\Lambda -\hat{%
			\phi}_{1,p_{2}})^{\beta _{2}}\geq \Lambda ^{-\beta _{2}}\Lambda ^{\beta
			_{2}}=1\text{ in }\Omega ,  \label{38}
	\end{equation}%
	for all $(u,v)\in \lbrack \underline{u}_{\Lambda },\overline{u}_{\Lambda
	}]\times \lbrack \underline{v}_{\Lambda },\overline{v}_{\Lambda }]$. Then,
	in view of (\ref{14}), we deduce from (\ref{37})-(\ref{38}) that $(%
	\underline{u}_{\Lambda },\underline{v}_{\Lambda })$ in (\ref{13*}) fulfill (%
	\ref{c2}) and therefore, $(\underline{u}_{\Lambda },\underline{v}_{\Lambda
	}) $ is a subsolution of $(\mathrm{P}).$ Using (\ref{7}), we derive that%
	\[
	\begin{array}{l}
		\underline{u}_{\Lambda }\geq \Lambda ^{-1}(\Lambda -\hat{\phi}%
		_{1,p_{1}})\geq \Lambda ^{-1}(\Lambda -||\hat{\phi}_{1,p_{1}}||_{\infty
		})\geq \Lambda ^{-1}(\frac{\Lambda}{2})=\frac{1}{2}
	\end{array}%
	\]%
	and 
	\[
	\begin{array}{l}
		\underline{v}_{\Lambda }\geq \Lambda ^{-1}(\Lambda -\hat{\phi}%
		_{1,p_{2}})\geq \Lambda ^{-1}(\Lambda -||\hat{\phi}_{1,p_{2}}||_{\infty
		})\geq \Lambda ^{-1}(\frac{\Lambda}{2})=\frac{1}{2}.%
	\end{array}%
	\]%
	Consequently, (\ref{c4}) is verified and thence, Theorem \ref{T7} guarantees 	the existence of a positive solution $(u,v)\in int \mathcal{C}_{+}^{1,\tau }(\overline{\Omega })\times int\mathcal{C}_{+}^{1,\tau }(\overline{\Omega }%
	),$	verifying (\ref{27}), which is unique once $\alpha _{2},\beta _{1}\in (-1,0)$ and assumption (\ref{c10}) is fulfilled. This completes the proof.
\end{proof}

\subsection{$L^{\infty }$- bound solutions}

In this part, we focus on subsolutions with zero trace condition on the
boundary $\partial \Omega $. Combined with (\ref{c1}), this type of
subsolutions bring to the forefront the singularity at the origin, generated
by negative exponents, which, until now, has been rather discreet due to the
strict positivity of the sub-supersolutions constructed. Obviously, this
will create additional difficulties in getting sub-supersolutions,
especially when the exponents are assumed all negative. Moreover,
singularities combined with zero trace subsolutions on the boundary $%
\partial \Omega $ constitute a significant barrier to get regular solutions.
At this point, note that we were not able to find in the literature the
Neumann counterpart of the regularity result for Dirichlet singular problems
in \cite[Lemma 3.1]{H}. This question remains open.

\begin{theorem}
	\label{T5} Assume (\ref{c1}) is fulfilled with $\alpha _{2},\beta _{1}>0$.
	Then, problem $(\mathrm{P})$ admits a solution $(u,v)\in
	W_{b}^{1,p_{1}}(\Omega )\times W_{b}^{1,p_{2}}(\Omega )$ such that%
	\begin{equation}
		(u,v)\in \lbrack \Lambda ^{-1}\phi _{1,p_{1}},\Lambda \hat{y}_{1}]\times
		\lbrack \Lambda ^{-1}\phi _{1,p_{2}},\Lambda \hat{y}_{2}],  \label{42}
	\end{equation}%
	provided $\Lambda >1$ is big enough.
\end{theorem}

\begin{proof}
	First, putting $(\underline{u},\underline{v}):=(\Lambda ^{-1}\phi
	_{1,p_{1}},\Lambda ^{-1}\phi _{1,p_{2}})$ note that (\ref{5}) implies that (%
	\ref{c1*}) is satisfied. Let $(u,v)\in W^{1,p_{1}}(\Omega )\times
	W^{1,p_{2}}(\Omega )$ within $[\Lambda ^{-1}\phi _{1,p_{1}},\Lambda \hat{y}%
	_{1}]\times \lbrack \Lambda ^{-1}\phi _{1,p_{2}},\Lambda \hat{y}_{2}]$. In
	view of (\ref{3}) and (\ref{6*}), we get%
	\begin{equation}
		\begin{array}{l}
			-\Delta _{p_{i}}\left( \Lambda ^{-1}\phi _{1,p_{i}}\right) +|\Lambda
			^{-1}\phi _{1,p_{i}}|^{p_{i}-2}\left( \Lambda ^{-1}\phi _{1,p_{i}}\right)
			=\lambda _{1,p_{i}}(\Lambda ^{-1}\phi _{1,p_{i}})^{p_{i}-1} \\ 
			\leq \lambda _{1,p_{i}}(\Lambda ^{-1}||\phi _{1,p_{1}}||_{\infty })^{p_{i}-1}%
			\text{ \ in }\Omega%
		\end{array}
		\label{41}
	\end{equation}%
	and 
	\begin{equation}
		-\Delta _{p_{i}}\left( \Lambda \hat{y}_{i}\right) +|\Lambda \hat{y}%
		_{i}|^{p_{i}-2}\left( \Lambda \hat{y}_{i}\right) =\Lambda ^{p_{i}-1}\text{\
			\ in }\Omega ,\text{ for }i=1,2.
	\end{equation}%
	By (\ref{c1}) one has%
	\begin{equation}
		(\Lambda ^{-1}\phi _{1,p_{1}})^{\alpha _{1}}+v^{\beta _{1}}\geq (\Lambda
		^{-1}\phi _{1,p_{1}})^{\alpha _{1}}\geq (\Lambda ^{-1}||\phi
		_{1,p_{1}}||_{\infty })^{\alpha _{1}}\text{ \ in }\Omega ,  \label{43}
	\end{equation}%
	\begin{equation}
		u^{\alpha _{2}}+(\Lambda ^{-1}\phi _{1,p_{2}})^{\beta _{2}}\geq (\Lambda
		^{-1}\phi _{1,p_{2}})^{\beta _{2}}\geq (\Lambda ^{-1}||\phi
		_{1,p_{2}}||_{\infty })^{\beta _{2}}\text{ \ in }\Omega ,  \label{43*}
	\end{equation}%
	while, from (\ref{4*}) and (\ref{5*}), it hold%
	\begin{equation}
		\begin{array}{l}
			(\Lambda \hat{y}_{1})^{\alpha _{1}}+v^{\beta _{1}}\leq (\Lambda \hat{y}%
			_{1})^{\alpha _{1}}+(\Lambda \hat{y}_{2})^{\beta _{1}}\leq (\Lambda \frac{%
				\hat{\phi}_{1,p_{1}}}{\hat{c}})^{\alpha _{1}}+(\Lambda \hat{c}\hat{\phi}%
			_{1,p_{2}})^{\beta _{1}} \\ 
			\leq (\Lambda \frac{\mu }{\hat{c}})^{\alpha _{1}}+(\Lambda \hat{c}||\hat{\phi%
			}_{1,p_{2}}||_{\infty })^{\beta _{1}}\leq \Lambda ^{\beta _{1}}((\frac{\mu }{%
				\hat{c}})^{\alpha _{1}}+(\hat{c}||\hat{\phi}_{1,p_{2}}||_{\infty })^{\beta
				_{1}})\text{ \ in }\Omega%
		\end{array}%
	\end{equation}%
	and%
	\begin{equation}
		\begin{array}{l}
			u^{\alpha _{2}}+(\Lambda \hat{y}_{2})^{\beta _{2}}\leq (\Lambda \hat{y}%
			_{1})^{\alpha _{2}}+(\Lambda \hat{y}_{2})^{\beta _{2}}\leq (\Lambda \hat{c}%
			\hat{\phi}_{1,p_{1}})^{\alpha _{2}}+(\Lambda \frac{\hat{\phi}_{1,p_{2}}}{%
				\hat{c}})^{\beta _{2}} \\ 
			\leq (\Lambda \hat{c}||\hat{\phi}_{1,p_{1}}||_{\infty })^{\alpha
				_{2}}+(\Lambda \frac{\mu }{\hat{c}})^{\beta _{2}}\leq \Lambda ^{\alpha
				_{2}}((\hat{c}||\hat{\phi}_{1,p_{1}}||_{\infty })^{\alpha _{2}}+(\frac{\mu }{%
				\hat{c}})^{\beta _{2}})\text{ \ in }\Omega .%
		\end{array}
		\label{41*}
	\end{equation}%
	Then, gathering (\ref{41})-(\ref{41*}) together, we conclude that (\ref{c2})
	and (\ref{c3}) are achieved for $\Lambda >0$ large enough. Hence, $\Lambda
	^{-1}(\phi _{1,p_{1}},\phi _{1,p_{2}})$ and $\Lambda (\hat{y}_{1},\hat{y}%
	_{2})$ are a sub-supersolutions of problem $(\mathrm{P})$. By (\ref{5}),
	assumption (\ref{c6}) is fulfilled. Consequently, owing to Theorem \ref{T8},
	there exists a solution $(u,v)\in W_{b}^{1,p_{1}}(\Omega )\times
	W_{b}^{1,p_{2}}(\Omega )$ of system $(\mathrm{P})$ verifying (\ref{42}).
\end{proof}

To deal with the case when all exponents are negative in (\ref{c1}), we
slightly modify the location area of potential solutions by moving the upper
limits of the rectangle formed through sub-supersolutions. To do so, assume $%
-1<\alpha _{2},\beta _{1}<0$ and let $\hat{z}_{1},\hat{z}_{2}\in
W_{b}^{1,p_{i}}(\Omega )$ be the unique solutions of Neumann Problems%
\begin{equation}
	-\Delta _{p_{1}}\hat{z}_{1}+|\hat{z}_{1}|^{p_{1}-2}\hat{z}_{1}=d(x)^{\beta
		_{1}}\text{ in }\Omega ,\text{ }\frac{\partial \hat{z}_{1}}{\partial \eta }=0%
	\text{ on }\partial \Omega ,  \label{44}
\end{equation}%
\begin{equation}
	-\Delta _{p_{2}}\hat{z}_{2}+|\hat{z}_{2}|^{p_{2}-2}\hat{z}_{2}=d(x)^{\alpha
		_{2}}\text{ in }\Omega ,\text{ \ }\frac{\partial \hat{z}_{2}}{\partial \eta }%
	=0\text{ on }\partial \Omega .  \label{44*}
\end{equation}

Note that, the Hardy-Sobolev type inequality (\ref{54}) guarantees that the
right-hand sides of (\ref{44}) and (\ref{44*}) belong to $%
W^{-1,p_{1}}(\Omega )^{*}$ and $W^{-1,p_{2}}(\Omega )^{*}$, respectively.
Consequently, Minty-Browder theorem (see, e.g., \cite[Theorem V.15]{B})
implies the existence and uniqueness of $\hat{z}_{i}\in W^{1,p_{i}}(\Omega )$
in (\ref{44}) and (\ref{44*}), $i=1,2$. Moreover, by weak comparison
principle (see \cite{ST}), it is readily seen that there is a constant $%
c_{1}>0$ such that 
\begin{equation}
	\hat{z}_{i}\geq c_{1}\hat{\phi}_{1,p_{i}}\text{ in }\Omega ,\text{ for }%
	i=1,2.  \label{45}
\end{equation}%
Next, we provide the $L^{\infty }$-bound of $\hat{z}_{i}.$

\begin{lemma}
	\label{L1}Under assumption 
	\begin{equation}
		0>\beta _{1},\alpha _{2}>\frac{-1}{N},  \label{br1}
	\end{equation}%
	solutions $\hat{z}_{1}\in W^{1,p_{1}}(\Omega )$ and $\hat{z}_{2}\in
	W^{1,p_{2}}(\Omega )$ of problems (\ref{44}) and (\ref{44*}) are bounded in $%
	L^{\infty }(\Omega )$.
\end{lemma}

\begin{proof}
	We only show the $L^{\infty }$-bound of $\hat{z}_{1}$ in (\ref{44}) because
	that of $\hat{z}_{2}$ in (\ref{44*}) can be justified similarly. Inspired by 
	\cite[Lemma 2]{AM}, for each $k\in \mathbb{N}$, set 
	\[
	A_{k}=\left\{ x\in \Omega :\hat{z}_{1}(x)>k\right\} .
	\]%
	It is readily seen that $|A_{k}|\rightarrow 0,$ as well as, 
	\begin{equation}
		\Vert d(x)^{\beta _{1}}\Vert _{L^{N}(A_{k})}\rightarrow 0\text{ \ as }%
		k\rightarrow \infty ,  \label{br3}
	\end{equation}%
	because $\hat{z}_{1}\in L^{1}(\Omega )$ and $d(x)^{\beta _{1}}\in
	L^{N}(\Omega )$ (due to (\ref{br1})).
	
	Test with $(\hat{z}_{1}-k)^{+}$ in (\ref{44}), it results in%
	\[
	\begin{array}{l}
		\int_{\Omega } |\nabla \hat{z}_{1} |^{p_{1}-2} \nabla \hat{z}_{1} \nabla (%
		\hat{z}_{1}-k)^{+} \mathrm{d}x+\int_{\Omega }\hat{z}_{1}^{p_{1}-1}(\hat{z}%
		_{1}-k)^{+}\,\mathrm{d}x=\int_{\Omega }d(x)^{\beta _{1}}(\hat{z}_{1}-k)^{+}%
		\mathrm{d}x.%
	\end{array}%
	\]%
	Thus%
	\begin{equation}
		\begin{array}{l}
			\int_{\Omega } |\nabla \hat{z}_{1} |^{p_{1}-2} \nabla \hat{z}_{1} \nabla (%
			\hat{z}_{1}-k)^{+} \mathrm{d}x+\int_{\Omega }\hat{z}_{1}^{p_{1}-1}(\hat{z}%
			_{1}-k)^{+}\,\mathrm{d}x \\ 
			\geq \int_{\Omega } |\nabla \hat{z}_{1} |^{p_{1}-2} \nabla \hat{z}_{1}
			\nabla (\hat{z}_{1}-k)^{+} \mathrm{d}x+\int_{\Omega }((\hat{z}%
			_{1}-k)^{+})^{p_{1}-1}(\hat{z}_{1}-k)^{+}\,\mathrm{d}x \\ 
			\geq \int_{A_{k}}|\nabla \hat{z}_{1}|^{p_{1}}\mathrm{d}x+\int_{A_{k}}(\hat{z}%
			_{1}-k)^{p_{1}}\,\mathrm{d}x%
		\end{array}
		\label{70}
	\end{equation}%
	while, by H\"{o}lder inequality together with Sobolev embedding $%
	W^{1,1}(\Omega )\hookrightarrow L^{\frac{N}{N-1}}(\Omega )$, one has%
	\begin{equation}  \label{71}
		\begin{array}{l}
			\int_{\Omega }d(x)^{\beta _{1}}(\hat{z}_{1}-k)^{+}\mathrm{d}%
			x=\int_{A_{k}}d(x)^{\beta _{1}}(\hat{z}_{1}-k)^{+}\mathrm{d}x \\ 
			\leq \Vert d^{\beta _{1}}\Vert _{L^{N}(A_{k})}\Vert (\hat{z}_{1}-k)^{+}\Vert
			_{L^{\frac{N}{N-1}}(A_{k})}\leq \Vert d^{\beta _{1}}\Vert
			_{L^{N}(A_{k})}\Vert (\hat{z}_{1}-k)^{+}\Vert _{L^{\frac{N}{N-1}}(\Omega )}
			\\ 
			\leq c_{p_{1}}\Vert d^{\beta _{1}}\Vert _{L^{N}(A_{k})}\Vert (\hat{z}%
			_{1}-k)^{+}\Vert _{W^{1,1}(\Omega )}=c_{p_{1}}\Vert d^{\beta _{1}}\Vert
			_{L^{N}(A_{k})}\Vert \hat{z}_{1}-k\Vert _{W^{1,1}(A_{k})}.%
		\end{array}%
	\end{equation}%
	Therefore, from (\ref{70})-(\ref{71}), it turn out that 
	\begin{equation}
		\int_{A_{k}}|\nabla \hat{z}_{1}|^{p_{1}}\mathrm{d}x+\int_{A_{k}}(\hat{z}%
		_{1}-k)^{p_{1}}\,\mathrm{d}x\leq c_{p_{1}}\Vert d^{\beta _{1}}\Vert
		_{L^{N}(A_{k})}\int_{A_{k}}(|\nabla \hat{z}_{1}|+(\hat{z}_{1}-k))\mathrm{d}x.
		\label{72}
	\end{equation}%
	Young's inequality yields%
	\begin{equation}
		\int_{A_{k}}(|\nabla \hat{z}_{1}|+(\hat{z}_{1}-k))\mathrm{d}x\leq
		\int_{A_{k}}|\nabla \hat{z}_{1}|^{p_{1}}\mathrm{d}x+\int_{A_{k}}(\hat{z}%
		_{1}-k)^{p_{1}}\mathrm{d}x+2|A_{k}|,  \label{br4}
	\end{equation}%
	which, combined with (\ref{72}), immediately leads to 
	\[
	\begin{array}{l}
		\int_{A_{k}}|\nabla \hat{z}_{1}|^{p_{1}}\mathrm{d}x+\int_{A_{k}}(\hat{z}%
		_{1}-k)^{p_{1}}\,\mathrm{d}x \\ 
		\leq c_{p_{1}}\Vert d^{\beta _{1}}\Vert _{L^{N}(A_{k})}\left(
		\int_{A_{k}}|\nabla \hat{z}_{1}|^{p_{1}}\mathrm{d}x+\int_{A_{k}}(\hat{z}%
		_{1}-k)^{p_{1}}\mathrm{d}x+2|A_{k}|\right) .%
	\end{array}%
	\]%
	Thereby, for $k$ large enough, the limit (\ref{br3}) implies 
	\[
	\begin{array}{l}
		\int_{A_{k}}|\nabla \hat{z}_{1}|^{p_{1}}\mathrm{d}x+\int_{A_{k}}(\hat{z}%
		_{1}-k)^{p_{1}}\,\mathrm{d}x\leq C_{1}|A_{k}|.%
	\end{array}%
	\]%
	Replacing the last inequality in (\ref{br4}), we achieve 
	\[
	\begin{array}{l}
		\int_{A_{k}}(|\nabla \hat{z}_{1}|+(\hat{z}_{1}-k))\mathrm{d}x\leq
		C_{2}|A_{k}|.%
	\end{array}%
	\]%
	Now, once again applying H\"{o}lder inequality and the above Sobolev
	embedding, it follows that 
	\[
	\begin{array}{l}
		\int_{A_{k}}(\hat{z}_{1}-k)\mathrm{d}x\leq |A_{k}|^{\frac{1}{N}}\Vert (\hat{z%
		}_{1}-k)\Vert _{L^{\frac{N}{N-1}}(A_{k})}\leq C|A_{k}|^{\frac{1}{N}%
		}\int_{A_{k}}(|\nabla \hat{z}_{1}|+(\hat{z}_{1}-k))\mathrm{d}x%
	\end{array}%
	\]%
	and so, 
	\[
	\int_{A_{k}}(\hat{z}_{1}-k)\mathrm{d}x\leq C_{3}|A_{k}|^{1+\frac{1}{N}}.
	\]%
	Then, owing to \cite[Lemma 5.1 , Chaper 2]{LU}, we conclude that there is $%
	\mathrm{K}>0$, independent of $\hat{z}_{1}$, such that 
	\[
	\begin{array}{l}
		\hat{z}_{1}(x)\leq \mathrm{K}\text{ a.e. in }\Omega ,%
	\end{array}%
	\]%
	showing that $\hat{z}_{1}\in L^{\infty }(\Omega )$. Similar argument
	produces that $\hat{z}_{2}\in L^{\infty }(\Omega )$ in (\ref{44*}). This
	ends the proof.
\end{proof}

\begin{theorem}
	\label{T9} Assume (\ref{c1}) is fulfilled with 
	\begin{equation}
		\max \{-1,-(p_{2}-1)\}<\alpha _{2}<0\text{ \ and \ }\max
		\{-1,-(p_{1}-1)\}<\beta _{1}<0.  \label{c0}
	\end{equation}%
	Then, problem $(\mathrm{P})$ admits a solution $(u,v)\in W^{1,p_{1}}(\Omega
	)\times W^{1,p_{2}}(\Omega )$ verifying%
	\begin{equation}
		(u,v)\in \lbrack \Lambda ^{-1}\phi _{1,p_{1}},\Lambda \hat{z}_{1}]\times
		\lbrack \Lambda ^{-1}\phi _{1,p_{2}},\Lambda \hat{z}_{2}],  \label{42*}
	\end{equation}%
	provided $\Lambda >0$ is big enough. Moreover, if (\ref{br1}) is fullfield,
	then $(u,v)\in W_{b}^{1,p_{1}}(\Omega )\times W_{b}^{1,p_{2}}(\Omega )$.
\end{theorem}

\begin{proof}
	It is readily seen that inequalities in (\ref{41}), (\ref{43}) and (\ref{43*}%
	) hold true even for $-1<\alpha _{2},\beta _{1}<0$. Then, $\Lambda
	^{-1}(\phi _{1,p_{1}},\phi _{1,p_{2}})$ is a subsolution of system $(\mathrm{%
		P})$ satisfying (\ref{c6}). Let us show that $\Lambda (\hat{z}_{1},\hat{z}%
	_{2})$ is a supersolution of $(\mathrm{P})$. By (\ref{45}), (\ref{5*}) and (%
	\ref{5}), we get%
	\begin{equation}
		\begin{array}{l}
			(\Lambda \hat{z}_{1})^{\alpha _{1}}+v^{\beta _{1}}\leq (\Lambda \hat{z}%
			_{1})^{\alpha _{1}}+(\Lambda ^{-1}\phi _{1,p_{2}})^{\beta _{1}} \\ 
			\leq (\Lambda c_{1}\hat{\phi}_{1,p_{1}})^{\alpha _{1}}+(\Lambda ^{-1}\phi
			_{1,p_{2}})^{\beta _{1}} \\ 
			\leq 1+(\Lambda ^{-1}c_{0}d(x))^{\beta _{1}} \\ 
			\leq ((\Lambda ^{-1}d(x))^{-\beta _{1}}+c_{0}^{\beta _{1}})(\Lambda
			^{-1}d(x))^{\beta _{1}} \\ 
			\leq (1+c_{0}^{\beta _{1}})(\Lambda ^{-1}d(x))^{\beta _{1}}\text{\ in }%
			\Omega ,%
		\end{array}
		\label{47}
	\end{equation}%
	and similarly%
	\begin{equation}
		\begin{array}{l}
			u^{\alpha _{2}}+(\Lambda \hat{z}_{2})^{\beta _{2}}\leq (\Lambda ^{-1}\phi
			_{1,p_{1}})^{\alpha _{2}}+(\Lambda \hat{z}_{2})^{\beta _{2}} \\ 
			\leq (\Lambda ^{-1}\phi _{1,p_{1}})^{\alpha _{2}}+(\Lambda c_{1}\hat{\phi}%
			_{1,p_{2}})^{\beta _{2}} \\ 
			\leq (\Lambda ^{-1}c_{0}d(x))^{\alpha _{2}}+1 \\ 
			\leq (c_{0}^{\alpha _{2}}+(\Lambda ^{-1}d(x))^{-\alpha _{2}})(\Lambda
			^{-1}d(x))^{\alpha _{2}} \\ 
			\leq (c_{0}^{\alpha _{2}}+1)(\Lambda ^{-1}d(x))^{\alpha _{2}}\text{\ in }%
			\Omega ,%
		\end{array}%
	\end{equation}%
	for all $(u,v)\in \lbrack \Lambda ^{-1}\phi _{1,p_{1}},\Lambda \hat{z}%
	_{1}]\times \lbrack \Lambda ^{-1}\phi _{1,p_{2}},\Lambda \hat{z}_{2}]$,
	provided that $\Lambda >0$ is large. On the other hand, (\ref{44}) and (\ref%
	{44*}) imply%
	\begin{equation}
		-\Delta _{p_{1}}(\Lambda \hat{z}_{1})+|\Lambda \hat{z}_{1}|^{p_{1}-2}(%
		\Lambda \hat{z}_{1})=\Lambda ^{p_{1}-1}d(x)^{\beta _{1}}\text{ in }\Omega
	\end{equation}%
	and%
	\begin{equation}
		-\Delta _{p_{2}}(\Lambda \hat{z}_{2})+|\Lambda \hat{z}_{2}|^{p_{2}-2}(%
		\Lambda \hat{z}_{2})=\Lambda ^{p_{2}-1}d(x)^{\alpha _{2}}\text{ in }\Omega .
		\label{47*}
	\end{equation}%
	Then, it turns out from (\ref{47})-(\ref{47*}) and (\ref{c0}) that (\ref{c3}%
	) is fulfilled for $\Lambda >0$ large enough. Hence, $\Lambda (\hat{z}_{1},%
	\hat{z}_{2})$ is\ a supersolution of problem $(\mathrm{P})$. Consequently,
	owing to Theorem \ref{T8}, there exists a solution $(u,v)\in
	W^{1,p_{1}}(\Omega )\times W^{1,p_{2}}(\Omega )$ of system $(\mathrm{P})$
	verifying (\ref{42*}). Moreover, according to Lemma \ref{L1} and (\ref{42*}%
	), we infer that $(u,v)\in W_{b}^{1,p_{1}}(\Omega )\times
	W_{b}^{1,p_{2}}(\Omega )$ once (\ref{br1}) is fulfilled. This ends the proof.
\end{proof}


\begin{thebibliography}{99}
	\bibitem{AC} C.O. Alves \& F.J.S.A. Correa, \emph{On the existence of
		positive solution for a class of singular systems involving quasilinear
		operators}, Appl. Math. Comput. 185 (2007), 727-736.
	
	\bibitem{ALVR} C.O. Alves, V.D. Radulescu, \emph{The Lane-Emden equation
		with variable double-phase and multiple regime}, Proc. Amer. Math. Soc. 148
	(2020), 2937-2952.
	
	\bibitem{Ben} R.D. Benguria, \emph{The Lane-Emden equation revisited},
	Contemporary Math. 327 (2003), 11-19.
	
	\bibitem{B} H. Br\'{e}zis, \emph{Analyse fonctionnelle theorie et
		applications}, Masson, Paris, 1983.
	
	\bibitem{C} S. Chandrasekhar, \emph{An Introduction to the Study of Stellar
		Structure}, Dover Publications, New York, 1957.
	
	\bibitem{CLM} S. Carl, V. K. Le \& D. Motreanu, \emph{Nonsmooth Variational
		Problems and Their Inequalities}, Springer Monogr. Math., Springer, New
	York, 2007.
	
	\bibitem{CF} E. Casas and L.A. Fernandez, \emph{A Green's formula for
		quasilinear elliptic operators}, J. Math. Anal. Appl., \ 142 (1989), 62-73.
	
	\bibitem{CRT} M.G. Crandall, P.H. Rabinowitz, L. Tartar, \emph{On a
		Dirichlet problem with a singular nonlinearity}, Comm. Partial Diff. Eqts. 2
	(1977), 193-222.
	
	\bibitem{DM} H. Dellouche, A. Moussaoui, \emph{Singular quasilinear elliptic systems with gradient dependence}, Positivity 26 (2022),
	doi:10.1007/s11117-022-00868-3.
	
	\bibitem{DM1} H. Didi, A. Moussaoui, \emph{Multiple positive solutions for a
		class of quasilinear singular elliptic systems}, Rend. Circ. Mat. Palermo,
	II. Ser 69 (2020), 977-994.
	
	\bibitem{DM2} H. Didi, B. Khodja, A. Moussaoui, \emph{Singular Quasilinear
		Elliptic Systems With (super-) Homogeneous Condition}, J. Sibe. Fede. Univ.
	Math. Phys. 13(2) (2020), 1-9.
	
	\bibitem{F} A. Farina, \emph{On the classification of solutions of the
		Lane-Emden equation on unbounded domains of }$%
	\mathbb{R}
	^{N}$, J. Math. Pures et App. 87 (5) (2007), 537-561.
	
	\bibitem{G} M. Ghergu, \emph{Lane-Emden systems with negative exponents}, J.
	Funct. Anal. 258 (2010), 3295-3318.
	
	\bibitem{GHM} J. Giacomoni, J. Hernandez, A. Moussaoui, \emph{Quasilinear
		and singular systems: the cooperative case}, Contemporary Math. 540 (2011).
	
	\bibitem{GST} J. Giacomoni, I. Schindler \& P. Takac, \emph{Sobolev versus H%
		\"{o}lder local minimizers and existence of multiple solutions for a
		singular quasilinear equation}, A. Sc. N. Sup. Pisa (5) 6 (2007), 117-158.
	
	\bibitem{GM} U. Guarnotta \& S.A. Marano, \emph{Infinitely many solutions to
		singular convective Neumann systems with arbitrarily growing reactions}, J.
	Diff. Eqts. 271 (2021), 849-863.
	
	\bibitem{H} D. D. Hai, \emph{On a class of singular p-Laplacian boundary
		value problems}, J. Math. Anal. Appl. 383 (2011), 619-626.
	
	\bibitem{K} M. A. Krasnoselskii, \emph{Topological Methods in the Theory of Nonlinear Integral Equations}, Pergamon Press, Oxford-London-Paris, 1964. (Translated from the Russian by A. H. Armstro)
	
	\bibitem{KM} B. Khodja \& A. Moussaoui, \emph{Positive solutions for
		infinite semipositone}$/$\emph{positone quasilinear elliptic systems with
		singular and superlinear terms}, Diff. Eqts. App. 8(4) (2016), 535-546.
	
	\bibitem{LU} O. A. Ladyzenskaja \& N. N. Ural'tzeva, \emph{Linear and
		Quasilinear Elliptic Equations}, Academic Press, New York, 1968.
	
	\bibitem{Lane} J. H. Lane, \emph{On the Theoretical Temperature of the Sun
		under the Hypothesis of a Gaseous Mass Maintaining Its Volume by Its
		Internal Heat and Depending on the Laws of Gases Known to Terrestrial
		Experiment}, American J. Sci. Arts, 50 (1870), 57-74.
	
	\bibitem{L} G. M. Lieberman, \emph{Boundary regularity for solutions of
		degenerate elliptic equations}, Nonl. Anal. 12 (1988), 1203-1219.
	
	\bibitem{MMP} D. Motreanu, V.V. Motreanu \& N.S. Papageorgiou, \emph{%
		Topological and Variational Methods with Applications to Nonlinear Boundary
		Value Problems}, Springer, New York (2014).
	
	\bibitem{MM3} D. Motreanu, A. Moussaoui, \emph{An existence result for a
		class of quasilinear singular competitive elliptic systems}, Appl. Math.
	Lett.\emph{\ }38 (2014), 33-37.
	
	\bibitem{MM2} D. Motreanu, A. Moussaoui, $\emph{A}$ $\emph{quasilinear}$ $%
	\emph{singular}$ $\emph{elliptic}$ $\emph{system}$ $\emph{without}$ $\emph{cooperative}$ $\emph{structure}$, Acta Math. Sci.\emph{\ }34 (B) (2014),
	905-916.
	
	\bibitem{MM1} D. Motreanu, A. Moussaoui, \emph{Existence and boundedness of
		solutions for a singular cooperative quasilinear elliptic system}, Complex
	Var. Elliptic Equ. 59 (2014), 285-296.
	
	\bibitem{M1} A. Moussaoui, \emph{Multiple solutions to Gierer-Meinhardt
		systems}, Disc. Continuous Dyn. Systs. 43 (7) (2023), 2835-2851.
	
	\bibitem{MKT} A. Moussaoui, B. Khodja, S. Tas, \emph{A singular
		Gierer-Meinhardt system of elliptic equations in}\textit{\ }$%
	\mathbb{R}
	^{N}$, Nonl. Anal. 71 (2009), 708-716.
	
	\bibitem{OK} B. Opic \& A. Kufner, \emph{Hardy-type inequalities}, Pitman
	Res. Notes Math., Longman, Harlow, 1990.
	
	\bibitem{ST} K. Sreenadh \& S. Tiwari, \emph{Global multiplicity results for 
	}$p(x)$\emph{-Laplacian equation with nonlinear Neumann boundary condition},
	
	\bibitem{S} M. Struwe, \emph{Variational Methods: Applications to nonlinear
		partial differential equations and Hamiltonian systems}, Springer-Verlag,
	Berlin, 1990.
	
	\bibitem{V} J.L. V\'{a}zquez, \emph{A strong maximum principle for some
		quasilinear elliptic equation}s, Appl. Math. Optimization 12 (1984), 191-202.
\end{thebibliography}
\end{document}